\documentclass[microtype]{gtpart}
\usepackage{tocloft}
\usepackage{mathtools}
\usepackage[all,cmtip]{xy}

\usepackage{hyperref}
\hypersetup{
colorlinks = true,
linkcolor = {blue},
urlcolor = {gray},
citecolor = {blue}
}
\usepackage[scr=rsfs]{mathalpha}
\usepackage{tikz}

\newcommand{\Rect}{\mathbf{R}}

\newcommand{\Mass}{\mathbb{M}}
\newcommand{\N}{\mathbb{N}}

\newcommand{\dist}{\mathrm{dist}}

\newcommand{\spt}{\mathrm{spt}\hspace{0.01cm}}

\newcommand{\eps}{\varepsilon}
\newcommand{\dV}{d_V\kern-1pt}
\newcommand{\dW}{d_W\kern-1pt}

\newcommand\res{\mathop{\hbox{\vrule height 7pt width .3pt depth 0pt
\vrule height .3pt width 5pt depth 0pt}}\nolimits}
\newcommand{\modk}{_{_{\mathbb{Z}_k}}\hspace{-0.1cm}}

\newcommand{\Mcal}{\mathcal{M}}

\newcommand{\Mbb}{\mathbb{M}}

\newcommand{\sing}{\mathrm{sing}}

\title{$\mathbb{Z}_k$-smooth approximation of mod $k$ integral cycles}

\author{Gianmarco Caldini and Anna Skorobogatova}
\givenname{}
\surname{}
\address{}
\email{}
\urladdr{}

\newtheorem{thm}{Theorem}[section]    
\newtheorem{corollary}[thm]{Corollary}
\newtheorem{pro}[thm]{Proposition}    
\newtheorem{lem}[thm]{Lemma}          
\theoremstyle{definition}
\newtheorem{definition}[thm]{Definition}    
\newtheorem{question}{Question}
\theoremstyle{definition}

\theoremstyle{remark}
\newtheorem{remark}[thm]{Remark}             
\newtheorem{example}[thm]{Example} 
\theoremstyle{definition}
\newtheorem{assumptions}[thm]{Assumptions}    

\usepackage{stmaryrd}

\begin{document}

\begin{abstract}  
We prove that every mod $k$ integral cycle $T$ of sufficiently high codimension in an oriented Riemannian manifold $\mathcal{M}$ can be approximated in flat norm by a cycle which is a $\mathbb{Z}_k$-smooth submanifold $\Sigma$ of nearly the same area, up to a singular set of codimension 5. If in addition $T$ belongs to a mod $k$ homology class admitting an embedded $\Z_k$-smooth manifold representative, then $\Sigma$ can be taken fully regular, up to the codimension one stratum of classical singularities. This extends the smooth approximation theory developed in \cite{ABCD, nonorientato} to cycles with coefficients in $\Z/k\Z$ in stable regimes. Moreover, we provide examples showing the impossibility of locally approximating minimizing cycles by smooth or $\mathbb{Z}_k$-smooth submanifolds, even in homology classes with smooth or $\mathbb{Z}_k$-smooth representatives.
\end{abstract}

\maketitle
\vspace{-1cm}
\setlength{\cftbeforesecskip}{7pt}
\tableofcontents


\section{Introduction}

\subsection{Motivation and main theorem}

In this article we are interested in developing the analog of the smooth approximation theory for integral and mod 2 cycles initiated in \cite{ABCD, nonorientato}, for cycles with coefficients in any group $\mathbb{Z}_k:=\Z/k\Z$ with $k \geq 3$, and in high codimension. Moreover, we provide examples of integral and mod $k$ minimizing cycles (in homology classes with smooth and $\mathbb{Z}_k$-smooth submanifolds representatives) which cannot be locally approximated by smooth and $\mathbb{Z}_k$-smooth submanifolds respectively\footnote{Without expressing any judgment, we remark that all new ideas and results in this article were developed before the ``technology shock'' given by LLM models and did not benefit from them in any possible way, \emph{cf.} \cite[Chapter 4]{tesi}. We decided to keep them separate from any possible further development, see Remark \ref{r:unstable}; we though believe, as of today, much more in the value of what we have not yet fully understood (or formalized), rather than just in the end of the story.}.

In dealing with homology classes with coefficients in $\Z_k$, the first natural question is to identify which object plays the role of a smooth manifold in such an environment. With the aim of obtaining profound results in surgery theory, Morgan and Sullivan developed in the early 70s a remarkable theory showing that the appropriate concept of smooth manifold in the $\mathbb{Z}_k$-coefficient setting is the geometric object they called \emph{$\mathbb{Z}_k$-manifold}, \emph{cf.} \cite{SullivanHaup, MorSul, MITnotes}.

A \emph{$\mathbb{Z}_k$-manifold} (or \emph{$\mathbb{Z}_k$-smooth manifold})  can be roughly defined as space obtained from an oriented smooth manifold $M$ with boundary identified with $k$ disjoint copies of an oriented smooth manifold $\delta M$, called the \emph{Bockstein of $M$}. Gluing these $k$ boundary components together gives a singular space which carries a fundamental $\mathbb{Z}_k$-homology class and the image of the Bockstein operation of this class for the coefficient sequence $$0 \rightarrow \mathbb{Z} \rightarrow \mathbb{Z} \rightarrow \mathbb{Z}_k \rightarrow 0$$ is indeed the integral class $[\delta M]$; in short, a $\mathbb{Z}_k$-manifold can be just thought of as a space which is nonsingular outside a codimension one submanifold where $k$ sheets join together, see Section \ref{s:Zkmanifolds} for the precise definition.

\begin{figure}[h]
    \centering
    \includegraphics[width=0.4\textwidth]{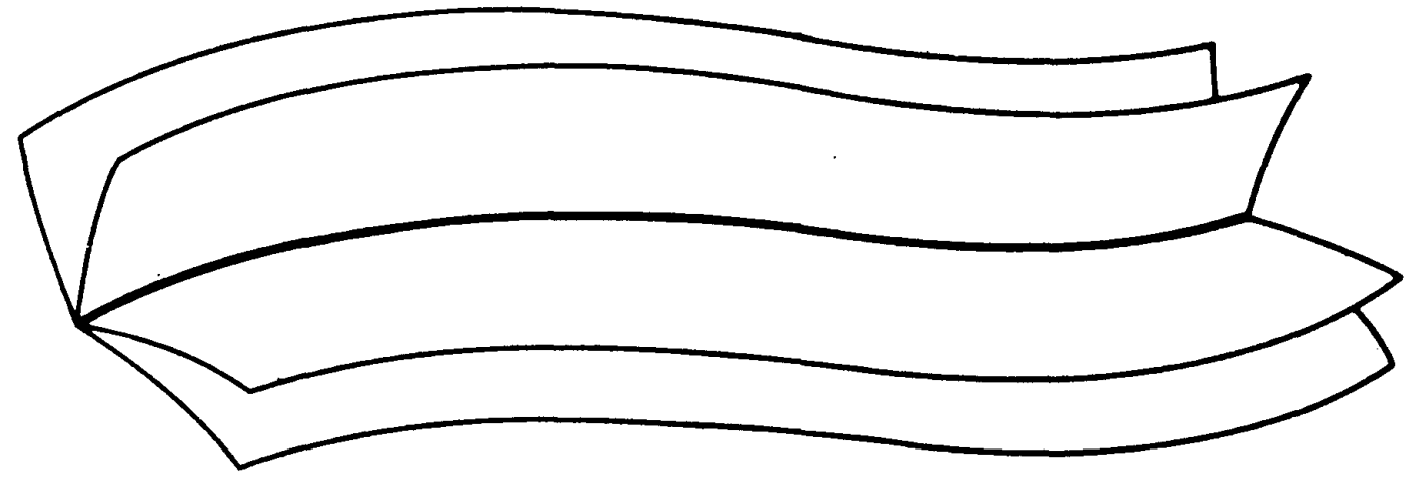}
    \caption{Part of a $\mathbb{Z}_4$-manifold, \emph{cf.} \cite[page 469]{MorSul}.}
\end{figure}

Much like the situation in the integral setting for smooth oriented manifolds, $\mathbb{Z}_k$-manifolds are not general enough to serve as a geometric model for $\mathbb{Z}_k$-homology when $k$ is odd: there exist mod $k$ classes which do not admit any $\mathbb{Z}_k$-manifold representative. For example, the generator of $H_8(K(\mathbb{Z},3), \mathbb{Z}_3)$ cannot be represented by the continuous image of a $\mathbb{Z}_3$-manifold because of a topological obstruction, \emph{cf.} \cite[page 78]{SullivanHaup}\footnote{It is also possible to find a 6-dimensional homology class which cannot be represented by a smooth $\mathbb{Z}_3$-manifold, see Section \ref{s:sharpness} and \cite{Brumfiel}.}. 

An immediate corollary of such an example is that it is not always possible to weakly approximate all integral cycles mod $k$ with a cycle induced by a $\mathbb{Z}_k$-smooth submanifold. This in addition shows that there are fundamental topological obstructions to obtaining generic regularity of minimizers away from the codimension 1 stratum of the singular set in $\mathbb{Z}_k$-homology classes, which is also analogous to the integral setting (\emph{cf.} Thom's example of a 7-dimensional integral homology class with no smooth representatives)\footnote{We refer to \cite{LiuZ, LiuZ2} for more detailed analyses on generic regularity for minimizing integral and mod 2 cycles.}. It is hence clear that, disregarding the codimension 1 stratum, the topological singularities of $\mathbb{Z}_k$-manifolds morally mimick those of oriented smooth manifolds; and indeed, this analogy persists for codimension one singularities of manifolds satisfying variational assumptions, as carefully described in \cite[Introduction \& Main Results]{MinterStanbury}.

This introduction motivates the following analytic and topological questions in the theory of integral currents mod an integer. 

\begin{question}\label{q:smooth-appr-when-smooth-rep}
Suppose $\tau$ is a $\mathbb{Z}_k$-homology class (of a smooth oriented manifold) admitting a smooth embedded $\mathbb{Z}_k$-manifold representative; is it possible to approximate any integral cycle mod $k$ in $\tau$ by a cycle induced by a $\mathbb{Z}_k$-smooth manifold?
\end{question}

\begin{question}\label{q:partial-smooth-appr}
More generally, what is the optimal partial regularity that a $\mathbb{Z}_k$-manifold must have in order to perform such an approximation without any additional assumption on its $\mathbb{Z}_k$-homology class?
\end{question}

The following theorem deals with such questions in high codimension, providing a positive answer to Question \ref{q:smooth-appr-when-smooth-rep} and -- in analogy with the integral setting \cite{ABCD} -- a codimension 5 singular set estimate for Question \ref{q:partial-smooth-appr}. In addition, in Section \ref{s:setwise} we provide examples of both integral and mod $k$ minimizing cycles (in homology classes with smooth and $\mathbb{Z}_k$-smooth submanifolds representatives) which cannot be locally approximated by smooth and $\mathbb{Z}_k$-smooth submanifolds respectively.

\begin{thm}\label{t:1}
Let $\mathcal{M}$ be a connected smooth closed oriented Riemannian manifold of dimension $m+n$ and assume $n\geq \min\{5,m+1\}$. Let $\varepsilon > 0$, $\tau$ an $m$-dimensional mod $k$ homology class in $H_{m}(\mathcal{M}, \mathbb{Z}_k)$, where $k\geq2$ is a natural number, and $T$ an integral cycle mod $k$ representing $\tau$. Then, there is a smooth triangulation $\mathcal{K}$ of $\mathcal{M}$ and an $m$-dimensional $\mathbb{Z}_k$-smooth submanifold $\Sigma$ of $\mathcal{M}\setminus \mathcal{K}^{m-5}$ (where $\mathcal{K}^{m-5}$ is the $m-5$-skeleton of $\mathcal{K}$) with the following properties.
\begin{enumerate}
\item The $m$-dimensional volume of $\Sigma$ does not exceed the mass of $T$ by more than $\varepsilon$, that is $\mathcal{H}^{m}(\Sigma) \leq \mathbb{M}(T) + \varepsilon$.
\item The mod $k$ integral cycle $\llbracket \Sigma \rrbracket\modk$ induced by $\Sigma$ is homologous to $T$ and there is a mod $k$ $m+1$-dimensional integral current $S$ in $\mathcal{M}$ such that $\partial S=\llbracket \Sigma \rrbracket\modk - T$ and $\mathbb{M}(S)< \varepsilon$.
\item If $\tau$ admits an embedded $\mathbb{Z}_k$-smooth manifold representative, then $\Sigma$ can be chosen to be a $\mathbb{Z}_k$-submanifold of $\mathcal{M}$.
\end{enumerate}
\end{thm}

\begin{remark}
The construction of the approximating $\Z_k$-smooth submanifold outside of a codimension 5 singular set in Theorem \ref{t:1} is optimal and matches (in high codimension) the dimensional estimates established in \cite{ABCD} on the singular set in the integral setting; see Section \ref{s:sharpness}. Although this is expected (given the definition of a $\mathbb{Z}_k$-smooth manifold), a geometric representation theory for singularities of mod $k$ integral cycles is needed to show it. This is achieved by Kreck's theory of stratifolds with $\mathbb{Z}_k$-coefficients, \emph{cf.} \cite{3tizi}, that we exploit together with Sullivan's geometric resolution of singularities \cite{SullivanLiverpool} to prove sharpness of our main result, see Theorem \ref{t:Thom_innatelysingular}.
\end{remark}

\begin{remark}
We remark that for $k=2$ the theorem is not in contrast with \cite{nonorientato}: while any $\mathbb{Z}_2$-manifold (with nonempty Bockstein) has the structure of an non-orientable smooth manifold, the converse is not true. In fact, $\mathbb{Z}_2$-manifolds admit more structural properties due to the orientation of the two manifolds joining at the Bockstein, see Remark \ref{r:Z2}.
\end{remark}

Some further remarks are now in order, mainly to comment on the need for sufficiently high codimension in Theorem \ref{t:1}.

\begin{remark}\label{r:unstable0}
The high codimension requirement $n\geq 5$ or, for $n<5$, $n\geq m+1$ in Theorem \ref{t:1} comes from working in the stable homotopy category of spectra; in particular, we need to be in the \emph{stable range} of dimensions. The intuition is that the complement of the $m-5$-skeleton of $\mathcal{M}$ has (homotopically) dimension $n+4$, so that when $n\geq 5$ we have enough connectivity to be in the stable dimensions for the correct ``classifying space'' for $\mathbb{Z}_k$-smooth embedded manifolds; analogous is the case for $n\geq m+1$, see Remark \ref{r:important}. Nonetheless, this does not seem to be a limitation of our approach, but rather a confirmation that embeddedness (or, in a sense, transversality) is a much more difficult property to obtain for singular manifolds rather than smooth ones. The fact that we are modeling $\mathbb{Z}_k$-manifolds, which are singular spaces (even though with well-behaved codimension 1 singularities), makes the whole theory much more complicated to deal with; this is the case, for instance, in the definition of the suitable ``classifying space for embedded $\mathbb{Z}_k$-smooth manifolds'' modeling such singularities and the global structure around them.
\end{remark}

\begin{remark}\label{r:unstable}
While Morgan-Sullivan mostly treat $\mathbb{Z}_k$-manifolds intrinsically, here we potentially have to face global obstructions related to how we embed the codimension 1 stratum of singularities in the ambient space $\mathcal{M}$ (which is locally modeled as $\mathbb{R}^{m-1} \times C(k)$, the product of an Euclidean spine with the cone over $k$ points) and the global restrictions imposed on the normal bundles of the sheets. If the codimension $n$ is high enough, these differences can be disregarded. It would be interesting to understand the unstable regimes also, depending on the conditions imposed on the normal bundle of the Bockstein. This would require defining suitable ``classifying spaces for $\mathbb{Z}_k$-manifolds'' with different restrictions on their normal bundles. Two other interesting directions are what can be said for smoothly \emph{immersed} $\mathbb{Z}_k$-manifolds\footnote{Coupling what has been developed here with Wells' theory of immersed cobordism in \cite{Wells} and compression theorems à la Rourke-Sanderson, \emph{cf.} \cite{RSI, RSII, RSIII}.}, \emph{cf}. \cite{MinterStanbury}, or if one is willing to allow some ``twisted'' structure in the sheeting or more general singularities of Baas-Sullivan type\footnote{According to the theory of \emph{twisted} $\mathbb{Z}_k$-manifolds and normal block bundles \emph{cf.} \cite{Rourke} and \cite[Example 3.1 (2)]{BRS}.}.
\end{remark}

\begin{remark}
Not only are Theorem \ref{t:1} and Remark \ref{r:unstable} of interest \emph{per se}, but they may also be compared to the (local) regularity theory of stable minimal (hyper)surfaces, area-minimizing currents mod $k$ with an immersed structure, and the various restrictions on the structure of \emph{classical} (or \emph{open book}) singularities. The latter represent the codimension 1 manifold of singularities given by the Bockstein, \emph{i.e.} singular points where the hypersurface has the structure of a finite number of smooth embedded submanifolds with a common identified boundary, \emph{cf.} \cite{IVW, dlandgang1}. In Wickramasekera’s regularity theory \cite{IVW} one assumes the absence of classical singularities, thus making it closer to a regularity theory for embedded stable minimal hypersurfaces. The high codimension assumptions in our global $\mathbb{Z}_k$-smooth approximation theory can be thought of as playing a similar role, where we disregard (in topologically stable regimes) the embeddability obstructions of classical singularities that may appear in low-codimension. However, allowing for classical singularities that are not necessarily embedded in the regularity theory allows for phenomena such as branch points, therefore making the theory significantly more complicated, both from the (local) analytic point of view, and the (global) topological one. We refer the reader to \cite{IVMW, MinDensity5/2, MinDensityQ+1/2, dlandgang2, SSS25, MinterStanbury, MX26, HSSS26} for more on the structure of singularities and sheeting theorems in the presence of branch points in the setting of stable minimal hypersurfaces and area-minimizing currents mod $k$.
\end{remark}

\begin{remark}
	Instead, the topological obstructions preventing one from representing mod $k$ homology classes as continuous images of $\mathbb{Z}_k$-manifolds and leading to the presence of the $m-5$-dimensional singular set do not vanish just by increasing the codimension $n$; they are not embeddability obstructions, but rather they represent \emph{innate singularities} in the sense of Sullivan \cite{SullivanLiverpool} (or, more precisely, in the sense of Thom \cite[Chapitre III]{Thom54}). This again confirms the analogy with the integral setting and differs drastically from the unoriented case in which all obstructions come from the possibility of deforming a singular submanifold to an embedded smooth one, \emph{cf.} \cite[Section 6]{nonorientato}. Hence, if one is willing to disregard any embeddability issue (\emph{i.e.} assuming sufficiently high codimension $n$) and focus the analysis of such innate singularities, Theorem \ref{t:1} provides the optimal estimates of the singular set; this comes as a corollary of the $4$-equivalence between the ``classifying space for embedded $\mathbb{Z}_k$-smooth manifolds'' $MSO \wedge M\mathbb{Z}_k$ and the Eilenberg-MacLane spectrum $H\mathbb{Z}_k$, see Proposition \ref{p:4-equivalence}.
\end{remark}

\subsection{Corollaries of the main result}
We present now some corollaries of Theorem \ref{t:1}. We will follow definitions and terminology of \cite{ABCD} and \cite[4.2.26]{Federerbook}, which will be briefly recalled in Section \ref{s:notationandpreliminaries}. Moreover, we will always rely on the following assumptions unless otherwise stated.

\begin{assumptions}\label{a:1}
 $\mathcal{M}$ is a connected, smooth, closed and oriented Riemannian manifold of dimension $m+n$, where $n,m \in \N\setminus \{0\}$ are arbitrary positive integers with $n\geq \min\{5,m+1\}$, $\tau$ is an element of the $m$-dimensional homology group $H_{m}(\mathcal{M}, \mathbb{Z}_k)$, where $k\geq 2$ is a positive integer, and $T$ is a mod $k$ integral current (hence a cycle) representing $\tau$.
\end{assumptions}

As usual, $\mathcal{K}^j$ will denote the $j$-dimensional skeleton of $\mathcal{K}$ and $\Sigma$ will denote either $\mathbb{Z}_k$-smooth $m$-dimensional closed embedded submanifolds of $\mathcal{M}$ or $\mathbb{Z}_k$-smooth $m$-dimensional embedded submanifolds of $\mathcal{M}\setminus \mathcal{K}^j$ (for some integer $j$) whose topological closure is contained in $\mathcal{K}^j$; We will denote by $\llbracket \Sigma \rrbracket\modk$ the integral cycle mod $k$ induced by $\Sigma$. 

\begin{definition}[\emph{$\mathbb{Z}_k$-smooth representability}]\label{d:smoothrepresentability}
Let $\mathcal{M}, \tau$ and $\Sigma$ be as in Assumption \ref{a:1} and assume $n\geq m+1$. We say that $\tau$ is \emph{representable by a smooth $\mathbb{Z}_k$-submanifold} (or that $\tau$ \emph{admits a smooth $\mathbb{Z}_k$ representative}) if there exists a smooth embedding\footnote{A smooth embedding of a $\Z_k$-manifold $\Sigma$ into a smooth manifold $\mathcal{M}$ has to be intended as a smooth embedding when considering $\Sigma$ as a smooth manifold with boundary passing to the quotient as a $\Z_k$-map, \emph{cf.} Section \ref{s:Zkmanifolds}.} $f:\Sigma \rightarrow \mathcal{M}$ such that $f_*[\Sigma]=\tau$, where $[\Sigma]\in H_m(\Sigma, \mathbb{Z}_2)$ is the fundamental class of $\Sigma$. Analogously, denoting by $x \in H^n(\mathcal{M},\mathbb{Z}_2)$ the Poincaré dual of $\tau$, we say that $x$ is realized by a smooth submanifold whenever $\tau$ is.
\end{definition}

Given Remarks \ref{r:unstable0} and \ref{r:unstable}, one may be instead interested -- on the opposite side -- in approximating any mod $k$ integral cycle with \emph{framed} embedded $\mathbb{Z}_k$-smooth submanifolds in $\mathcal{M}$, \emph{i.e.} smooth $\mathbb{Z}_k$-manifolds with a framing of the normal bundle. Thanks to Proposition \ref{p:poly_approx_prescribedsing} below and the fact that Moore spaces $M(\Z_k,n)$ play the role of ``classifying spaces for framed embedded $\mathbb{Z}_k$-smooth manifolds'', we obtain the following approximation theorem\footnote{Then, depending on the homotopy groups of Moore spaces (which clearly may vary depending on $k$ and $n$) it may even be possible to relax the assumption that $\tau$ admits a framed $\Z_k$-manifold representative, provided a certain lower dimensional skeleton of singularities is allowed in $\Sigma$.}.

\begin{thm}\label{t:Mooreapproximation}
Let $\eps> 0$ and let $\mathcal{M}, \tau, T$ be as in Assumption \ref{a:1}, and assume further that $\tau$ admits an embedded framed $\mathbb{Z}_k$-smooth manifold representative, $k$ is odd and $n\geq3$. Then, there is an embedded framed $\mathbb{Z}_k$-smooth submanifold $\Sigma$ of $\mathcal{M}$ with the following properties.
\begin{enumerate}
\item $\mathcal{H}^{m}(\Sigma) \leq \mathbb{M}(T) + \varepsilon$.
\item The mod $k$ integral cycle $\llbracket \Sigma \rrbracket\modk$ is $\varepsilon$-close to $T$ in mod $k$ flat norm, that is, \newline$ \mathbb{F}^k(\llbracket \Sigma \rrbracket\modk - T) \leq \varepsilon$.
\end{enumerate}
\end{thm}

As a byproduct of the proof of Theorem \ref{t:1} we obtain a $\Z_k$-smooth approximation result for low-dimensional cycles in high codimension.

\begin{corollary}
Let $\mathcal{M}, \tau$ be as in Assumption \ref{a:1}. Then, $\tau$ admits a smoothly embedded $\mathbb{Z}_k$-manifold representative whenever $m \in \{1,2,3,4\}$, and hence Theorem \ref{t:1} (3) applies. 
\end{corollary}

We remark that in high codimension $n\geq m+1$, from \cite[page 465]{MorSul}, every $m$-dimensional mod $k$ class admits an embedded $\Z_k$-manifold representative if $k=2^r,$ for any $r \in \N$.

Two other immediate corollaries of our constructions are the following. Denoting $\mathcal{Z}_m(\mathcal{M}, \mathbb{Z}_k)$ the set of $m$-dimensional mod $k$ integral cycles with mod $k$ support in $\mathcal{M}$ and by $\mathcal{Z}_{m, Lip}(\mathcal{M},\mathbb{Z}_k)$ the set of $m$-dimensional mod $k$ Lipschitz cycles with mod $k$ support in $\mathcal{M}$. In analogy with \cite{ABCD, nonorientato}, Theorem \ref{t:1} shows the absence of the \emph{Lavrentiev gap phenomenon} for the homological mod $k$ Plateau problem in the absence of topological obstructions.

\begin{thm}[Absence of Lavrentiev gaps]\label{t:Lavrentiev}
Let $\mathcal{M}, \tau, T$ and $\Sigma$ be as in Assumption \ref{a:1}, assume further that $\tau$ is nontrivial and define the following quantities:
\begin{align*}
&\mathbb{M}_{T}:=\min\{\mathbb{M}(T) : T \in \mathcal{Z}_m(\mathcal{M}, \mathbb{Z}_k) \cap \tau\},\\
&\mathbb{M}_{P}:=\inf\{\mathbb{M}(P) : P \in \mathcal{Z}_{m, Lip}(\mathcal{M}, \mathbb{Z}_k) \cap \tau\},\\
&\mathbb{M}_{\Sigma}:=\inf\{\mathcal{H}^m(\Sigma) : \llbracket\Sigma\rrbracket\modk \in \tau, \textrm{ $\Sigma$ is a $\mathbb{Z}_k$-manifold in $\mathcal{M} \setminus \mathcal{K}^{m-5}$ for some triangulation $\mathcal{K}$}\},\\
&\mathbb{M}_{{\rm Reg}} := \inf\{\mathcal{H}^m(\Sigma) : \llbracket\Sigma\rrbracket\modk \in \tau, \textrm{ $\Sigma$ is a $\mathbb{Z}_k$-manifold in $\mathcal{M}$}\}\, .
\end{align*}
Then, $\mathbb{M}_{T}= \mathbb{M}_{P}= \mathbb{M}_{\Sigma}$, and $\mathbb{M}_T = \mathbb{M}_{{\rm Reg}}$ when $\tau$ is representable by an embedded $\mathbb{Z}_k$-smooth submanifold.
\end{thm}

Finally, we also immediately obtain the following approximation theorem with integral cycles mod $k$ with prescribed singularities.

\begin{thm}[Approximation by cycles with prescribed singular sets]\label{t:prescribed}
Let $\mathcal{M}, \tau$ and $T$ be as in Assumption \ref{a:1}. Then, there is a sequence of smooth triangulations $\mathcal{K}_j$ of $\mathcal{M}$ and a sequence of embedded $\mathbb{Z}_k$-smooth $m$-dimensional submanifolds $(\Sigma_j)_j$ in $\mathcal{M}\setminus \mathcal{K}_j^{m-5}$ such that
\begin{itemize}
\item[(a)] $\llbracket \Sigma_j\rrbracket\modk \rightarrow T$ in the mod $k$ flat topology,
\item[(b)] $\lim_{j\rightarrow \infty}\mathcal{H}^{m}(\Sigma_j) = \mathbb{M}(T)$,
\item[(c)]$\partial \llbracket \Sigma_j\rrbracket\modk =0$ and $\llbracket \Sigma_j\rrbracket\modk$ is in the same homology class as $T$.
\end{itemize}
\end{thm}

\subsection{Overview of the proof}

Our approach builds on the new theory introduced in \cite{ABCD}, later adapted in \cite{nonorientato} to deal with integral cycles mod 2, and we encourage the reader to review the strategy of the proofs developed therein before concentrating on the novelties presented here. 

In this paper we introduce new concepts and methods that advance the techniques of \cite{ABCD} mainly in three directions: we develop rather sophisticated topological tools to extend the smooth approximation theory to cycles with coefficients in $\Z/k\Z$, using Morgan-Sullivan's theory of $\Z_k$-manifolds \cite{MorSul}; we noticeably simplify many arguments of \cite{ABCD}, in particular eliminating the need for considering smooth neighborhoods of skeleta and relative homology classes in their complements, \emph{cf.} \cite[Section 3.1]{ABCD}; and, moreover, we remark on the impossibility of locally approximating integral and mod $k$ integral cycles, highlighting in particular the limitations of such smooth approximation theories in general codimension.

Let us now comment in more detail on the important differences and additional difficulties we face in proving Theorem \ref{t:1}. In particular, while the geometric measure theory propositions in Section \ref{s:GMT} are just simple adaptations of previous results (like the ``codimension 2 smooth approximation theorem'' and the ``squeezing lemma'' in \cite[Propositions 4.1 and 4.3]{ABCD}), there are a number of topological difficulties that we need to face and that come, roughly speaking, as a consequence of the following fact: dealing with $\Z_k$-manifolds is much more delicate since we are dealing with singular spaces, rather than smooth ones. We face this by introducing new tools and topological techniques (mostly from stable homotopy theory) that we list below.

\vspace{0.25cm}
\noindent\textbf{New difficulties, new ingredients and differences with the smooth approximations}
\vspace{-0.15cm}

\vspace{0.25cm}
\noindent \textit{1.3.1 Geometric measure theory propositions.}
\vspace{-0.15cm}

First of all, we need to adapt the codimension 2 smooth approximation theorem of \cite[Proposition 4.1]{ABCD}: we want to approximate every mod $k$ cycle with an embedded $\mathbb{Z}_k$-smooth submanifold outside of the $m-2$-dimensional skeleton $\mathcal{K}^{m-2}$ of a smooth triangulation $\mathcal{K}$ of $\mathcal{M}$, see Proposition \ref{p:poly_approx_prescribedsing}. This requires a few additional steps compared to the analogous theorem in the integral setting: while there is no need to resolve the codimension 1 singularities, in order to obtain a $\Z_k$-smooth manifold structure outside of $\mathcal{K}^{m-2}$ we additionally need to split them into groups, each consisting of exactly $k$ $m$-dimensional faces joining together. This is not complicated, but must be taken care of in order to obtain a map $F$ from $\mathcal{M}\setminus \mathcal{K}^{m-2}$ with values in a suitable classifying space for embedded $\Z_k$-smooth manifolds.

\vspace{0.25cm}
\noindent \textit{1.3.2 Borel-Moore homology.}
\vspace{-0.15cm}

The first important novelty is the much cleaner option of working with a Thom's criterion for manifolds that are not necessarily compact, to avoid dealing with smooth neighborhoods of skeleta and relative Thom's criteria as in \cite[Section 3.1 and Theorem 2.6]{ABCD}. This requires the introduction of Borel-Moore homology and a Thom representability criterion valid for Borel-Moore homology, \emph{cf.} Section \ref{s:spectra}, allowing us to work with homology classes and properly embedded submanifolds in $\mathcal{M}\setminus \mathcal{K}^{m-2}$ and $\mathcal{M}\setminus \mathcal{K}^{m-5}$ directly. We note that this method can be adopted in \cite{ABCD, nonorientato} as well, noticeably simplifying many technical steps.

\vspace{0.25cm}
\noindent \textit{1.3.3 Stable homotopy category of spectra and the classifying space for embedded $\Z_k$-smooth manifolds.}
\vspace{-0.15cm}

We now describe the core differences with respect to the main argument in \cite{ABCD, nonorientato}, which are of a topological nature. First, instead of working with maps with values in the Thom spaces of the universal $n$-plane bundles $MSO(n)$ or $MO(n)$, we need to deal with the correct ``classifying space for embedded $\Z_k$-manifolds'' which, roughly speaking, is a space that encodes all infinitesimal models of what a $\Z_k$-manifold can look like. Specifically, it must model codimension 1 singularities where $k$ oriented sheets join together. Building on Morgan-Sullivan's theory in \cite{MorSul}, this classifying space is represented by a suitable ``$n^{th}$-space'' inside $MSO \wedge M\Z_k$, the latter being the smash product of the oriented Thom spectrum $MSO$ and the Moore spectrum $M\Z_k$. In particular, we rely on stable homotopy theory and the theory of spectra to prove that the map from $MSO \wedge M\Z_k$ with values in the Eilenberg-MacLane spectrum $H\Z_k$ representing the mod $k$ Thom class is a $4$-equivalence of spectra, \emph{cf.} Proposition \ref{p:4-equivalence}. This is morally the analog of the crucial $n+4$-homotopy equivalence between $MSO(n)$ and the Eilenberg-MacLane space $K(\Z,n)$ of \cite[Lemma 5.2]{ABCD} in this setting. We provide a brief introduction to the stable homotopy category of spectra in Section \ref{s:Zkmanifolds}. 

Nevertheless, this suitable $n^{th}$-space (formally defined as $\Omega^{\infty}\Sigma^{n}(MSO \wedge M\Z_k)$, which is the infinite loop space of the $n$-shifted $\Omega$-spectrum equivalent to $MSO \wedge M\Z_k$) is just an abstract space and no transversality theorem can be applied directly with it, simply because it is not the Thom space of some vector bundle. However, in stable regimes (\emph{i.e.} when the codimension $n$ is high enough), we show that this suitable $n^{th}$-space of $MSO \wedge M\Z_k$ can be identified with $MSO(n-1) \wedge M(\Z_k,1)$, \emph{cf.} Lemma \ref{l:MSOtoOmega}. This allows for a refined Thom criterion for embedded $\Z_k$-smooth submanifolds to be carried out directly on $\mathcal{M}$ (and not only on $\mathcal{M}\times \R^{\ell}$ for $\ell$ large, as done in \cite{MorSul} for instance), with homotopy classes of maps taking values in $MSO(n-1) \wedge M(\Z_k,1)$, \emph{cf.} Theorem \ref{t:criterionmodk}. In this definition, we remark that the Moore space $M(\Z_k,1)$ has to be thought intuitively as the space modeling the codimension 1 stratum of singularities, near which the $\Z_k$-manifold is the product of an $m-1$-dimensional smooth oriented embedded manifold and the cone on $k$-points. We refer to Remark \ref{r:important} for more precise comments on why, in stable regimes, $MSO(n-1) \wedge M(\Z_k,1)$ can be identified as the correct classifying space for embedded $\Z_k$-smooth manifolds.

In addition, we mention that such a Thom criterion -- and in particular the transversality argument -- is more complicated than the one for smooth submanifolds: it relies on a delicate two-stage transversality procedure that is only mentioned and not proved in \cite{MorSul}; here we develop a clear proof of it as part of Theorem \ref{t:criterionmodk}.

\vspace{0.25cm}
\noindent \textit{1.3.4 Optimality of the codimension 5 singular set.}
\vspace{-0.15cm}

In Section \ref{s:sharpness} we rely on Kreck's theory of stratifolds and its recent generalization to $\Z_k$-coefficients \cite{3tizi} to show that the codimension 5 singular set in Theorem \ref{t:1} is optimal. In fact, the latter theory provides a geometric model allowing to represent every $\mathbb{Z}_k$-homology class. In particular, this allows for a fine geometric description of those singularities obstructing representability of $\Z_k$-homology classes by $\Z_k$-smooth manifolds. In Theorem \ref{t:Thom_innatelysingular} we prove this optimality via Sullivan's geometric resolution of singularities \cite{SullivanLiverpool} for $\Z_3$-manifolds. In particular, we manufacture an oriented manifold $\mathcal{M}$ with a mod $3$ homology class of sufficiently high codimension that does not admit a $\Z_3$-manifold representative outside of the $m-6$-skeleton $\mathcal{K}^{m-6}$ of any smooth trangulation $\mathcal{K}$ of $\mathcal{M}$.

\vspace{0.25cm}
\noindent \textit{1.3.5 Impossibility to obtain a setwise approximation theorem.}
\vspace{-0.15cm}

In Section \ref{s:setwise}, we construct an example of an area-minimizing cycle $T$ in a nontrivial integral homology class (admitting a smooth embedded representative) of an oriented Riemannian manifold $\mathcal{M}$ whose support cannot be locally approximated by a smooth submanifold. We then adapt the construction to show that the same phenomenon can be replicated for mod $k$ minimizers of nontrivial $\Z_k$-homology classes with $\Z_k$-smooth manifold representatives.

\vspace{0.25cm}
\noindent\textbf{Previous literature}
\vspace{0.1cm}

The theory of integral currents was introduced by Federer and Fleming in their celebrated article \cite{FedererFleming60} to provide a suitable framework to solve the oriented Plateau problem without topological or dimensional restrictions, but requiring the domain of integration to be oriented and homology classes to have integer coefficients. Their theory of integration was later extended to different classes of surfaces firstly by Ziemer in \cite{Ziemer} (over chains with coefficients in $\mathbb{Z}_2$) and, subsequently, in much broader generality by Fleming in \cite{Fleming66}, over chains with coefficients in a finite abelian group $G$ (and, in particular, in $\mathbb{Z}_k$ for any natural number $k\geq 2$). This allowed for the development of a general existence theory for mass minimization problems over such classes, maintaining important combinatorial properties and a basic level of (measure theoretic) regularity. Another extension worth mentioning came with the work of White \cite{White90s}, in which the main rectifiability results were generalized to chains with coefficients in a general normed abelian group\footnote{More precisely, one may characterize those complete normed abelian coefficient groups for which finite-mass flat chains are rectifiable as being those that do not contain any non-constant finite-length continuous curve.}. 

Many regularity properties have been later derived for integral currents mod $k$. In particular, an extended literature has been developed for local regularity of mod $k$ minimizers, with different results depending on both the codimension and the integer $k$. In particular, in codimension $n=1$, the regularity theory for mod $k$ minimizing hypercurrents started from the early works by Taylor in \cite{TaylorMod3} for two-dimensional mod $3$ minimizers in $\R^3$, which provides a complete local description of the singular set as real analytic arcs where three regular (multiplicity one) sheets of the current meet at equal angles\footnote{When coupled with the regularity bootstrapping of Kinderlehrer-Nirenberg-Spruck \cite{KNS}.}. Then, White in \cite{WhiteMod4} considered the case of mod $4$ minimizing hypersurfaces, showing a local characterization via a decomposition into a pair of mod $2$ minimizers, which in particular yields a structure consisting of smooth immersed submanifolds outside of a closed set of zero $\mathcal{H}^{m-2}$-measure. Furthermore, in White \cite{White86} considered mod $k$ minimizing hypersurfaces for $k$ odd, showing that they are smooth locally around any point with density below $k/2$ (which is the highest density a.e.); this in turn establishes that their singular set has Hausdorff dimension at most $m-1$. 

The first partial regularity theorem for area-minimizing integral currents mod $k$ of arbitrary dimension, codimension and modulus $k$ was proved recently by De Lellis, Hirsch, Marchese and Stuvard in \cite{dlandgangCPAM, dlandgangGAFA}. Since then, a general refined analysis of the structure of the singular set of minimizers in this setting has been carried out: by the combined effort of \cite{dlandgang1, dlandgang2, dlandgang3, IVMW, Sk-modp} we now know that $m$-dimensional area-minimizing integral currents mod $k$ in smooth Riemannian manifolds of dimension $m+1$ have an interior singular set which is, up to a countably $(m-2)$-rectifiable set, a $C^{1,\alpha}$-submanifold\footnote{The analogue of the mod $3$ higher regularity result \cite{KNS} is established by Krummel in \cite{KrummelHigherReg}, which shows in fact that $C^{1,\alpha}$ regularity can be upgraded to the same regularity as the ambient Riemannian manifold.} of dimension $m-1$ at which, locally, $N\leq k$ regular sheets of the current join transversally, each sheet counted with a positive integer multiplicity $h_i$ so that $\sum_ih_i=k$. When the current is two-dimensional, one can in addition obtain a more precise structure close to branch points, as recently established by Skorobogatova-Spolaor-Stuvard in \cite{SSS25}, and subsequently in Krummel-Minter-Wickramasekera in \cite{KMW} as a consequence of the branched stratification theory therein combined with the regularity theory in \cite{IVMW}.

The fine structure of the singular set was first studied also in codimension higher than 1 only recently by De Lellis-Minter-Skorobogatova in \cite{modq}, showing that the singular set consists locally of an $(m-1)$-dimensional $C^{1,\alpha}$-submanifold, up to a countably $(m-2)$-rectifiable set. For two-dimensional currents of higher codimension, one can obtain more precise sheeting results close to classical singularities and branch points, analogous to the aforementioned ones for hypersurfaces; see \cite{HSSS26}. Hence, in all such regularity results the clear pattern is that of recovering the well-known regularity theory for mass-minimizing integral currents, up to the codimension 1 smooth manifold of singularities.

Instead, the only \emph{a priori} (global) approximation-type result for integral currents mod $k$ was the classical mod $k$ deformation theorem \cite[(4.2.9)$^{\nu}$]{Federerbook} and, as a corollary, the mod $k$ strong polyhedral approximation theorem, \textit{cf.} \cite[(4.2.20)$^{\nu}$, (4.2.21)$^{\nu}$]{Federerbook}, both developed as simple adaptations of the ones for integral currents. In analogy with what has been done for integral cycles and integral cycles mod 2, it is hence natural to expect that similar ``smooth'' approximation theorems to the ones developed in \cite{ABCD, nonorientato} hold for integral cycles mod $k$, where the concept of smooth manifold is replaced by that of $\mathbb{Z}_k$-smooth manifold.

This article provides such $\mathbb{Z}_k$-smooth approximation every time the mod $k$ homology class admits a smooth embedded $\Z_k$-manifold representative and proves an optimal codimension-5 singular set estimate for general mod $k$ homology classes in stable regimes. Finally, in Section \ref{s:setwise} we provide examples of integral and mod $k$ minimizing cycles (in homology classes with smooth and $\mathbb{Z}_k$-smooth submanifolds representatives) which cannot be locally approximated by smooth and $\mathbb{Z}_k$-smooth submanifolds respectively.


The rest of the paper is organized as follows. In Section \ref{s:notationandpreliminaries} we briefly recall the main notation and results in the theory of integral currents mod $k$, in stable homotopy theory, Borel-Moore homology and $\Z_k$-manifolds; in particular, we prove the key mod $k$ Thom criterion in Theorem \ref{t:criterionmodk}, together with Lemma \ref{l:MSOtoOmega}. In Section \ref{s:GMT} we prove the adapted mod $k$ version of the two main technical propositions from geometric measure theory, \emph{i.e.} Proposition \ref{p:poly_approx_prescribedsing} and Proposition \ref{p:squash}. Finally, Section \ref{s:proofofmain} is dedicated to the proof of Theorem \ref{t:1}, Section \ref{s:sharpness} shows the optimality of the construction and in Section \ref{s:setwise} we construct examples showing the impossibility of improving Theorem \ref{t:1} to a setwise approximation theorem.

\subsection*{Acknowledgements}
GC would like to thank Martin Bendersky for useful and enjoyable discussions about spectra and the connectivity of Thom spaces. GC's research is funded by the European Research Council (ERC) through CoG 101169953 ``RISE”\footnote{Views and opinions expressed are however those of the authors only and do not necessarily reflect those of the European Union or the European Research Council.}. A.S.~is grateful for the generous support of Dr.~Max R\"ossler, the Walter Haefner Foundation, and the ETH Z\"urich Foundation. This research was conducted during the period A.S.~served as a Clay Research Fellow.

\section{Notation and preliminary results}\label{s:notationandpreliminaries}

We recall here the main definitions and relevant notation. We remark that throughout the whole article $k$ will denote a positive integer greater or equal than 2; in addition, for a fixed $k$, we adopt the standard notation $\mathbb{Z}_k$ to denote the cyclic group $\mathbb{Z}/k\mathbb{Z}$ of order $k$.

\subsection{Theory of integral currents mod $k$}\label{s:currentsmodk}

We briefly recall the theory of mod $k$ integral currents; as a standard reference we refer to \cite[4.2.26]{Federerbook} or, more historically, to Fleming in \cite{Fleming66}.

We denote by $\mathbf{R}_d(\mathbb{R}^N)$ the space of $d$-dimensional integer rectifiable currents, by $\mathbf{I}_d(\R^N)$ the subgroup of $d$-dimensional integral currents and by $\mathbf{F}_d(\mathbb{R}^N)$ the space of $d$-dimensional integral flat chains in $\mathbb{R}^N$, that is $$\mathbf{F}_d(\mathbb{R}^N):=\{T + \partial S : T \in \mathbf{R}_d(\mathbb{R}^N), \,S \in \mathbf{R}_{d+1}(\mathbb{R}^N)\}.$$ If $M \subset \mathbb{R}^N$ is a compact Lipschitz neighborhood retract, then $\mathbf{R}_d(M)$ will denote the space of $T \in \mathbf{R}_d(\mathbb{R}^N)$ with support $\spt(T)\subset M$; analogously for $\mathbf{F}_d(M)$ and $\mathbf{I}_d(M)$.

\noindent
For $T \in\mathbf{F}_d(M)$, the (integral) flat norm of $T$ is defined as $$\mathbb{F}(T):=\inf \{ \Mass(R) + \Mass(S) \mid T=R +\partial S, \, R \in \mathbf{I}_d(M),\, S \in \mathbf{I}_{d+1}(M)\}.$$

\noindent
For $T \in \mathbf{F}_d(M)$ we recall that its mod $k$ flat norm can be defined as $$\mathbb{F}^k(T):= \inf\{\mathbb{F}(T+kQ) : Q \in \mathbf{F}_d(M)\},$$ its mod $k$ mass as $$\mathbb{M}^k(T):= \inf\{\mathbb{M}(T + kQ) : Q \in \mathbf{F}_d(M)\}$$ and its mod $k$ support as $$\spt^k(T):= \bigcap \{\spt(R) : R \in \mathbf{F}_d(M), \mathbb{F}^k(T-R)=0\}.\footnote{We will usually drop the index $k$ when this is clear from the context.}$$ 

\noindent
Thanks to the beautiful work of Young \cite{Youngnonorientability}, we can define the space of mod $k$ flat chains in $M$ is defined as the quotient group $\mathbf{F}_d(M,\mathbb{Z}_k):=\mathbf{F}_d(M)/k\mathbf{F}_d(M)$ and if $T \in \mathbf{F}_d(M)$, we will denote with $[T]$ or $T$ mod $k$ the coset in $\mathbf{F}_d(M,\mathbb{Z}_k)$ containing $T$. Analogously, we can define the space of integer rectifiable mod $k$ currents and of mod $k$ integral currents, respectively denoted by $$\mathbf{R}_d(M,\mathbb{Z}_k):=\mathbf{R}_d(M)/k\mathbf{R}_d(M), \quad \mathbf{I}_d(M,\mathbb{Z}_k):=\mathbf{I}_d(M)/k\mathbf{I}_d(M).$$ 

\noindent
If $T$ is flat chain mod $k$ with an integer coefficient polyhedral chain as a representative, then $T$ will be called a polyhedral mod $k$ chain; analogously for Lipschitz mod $k$ chains.

In particular, we remark that an integer rectifiable current $T =\llbracket E,\tau, \theta \rrbracket \in \Rect_d(M)$ is called \emph{representative mod $k$} whenever $|\theta(x)|\leq k/2$ for $\|T\|$-$a.e.$ point. Hence, every integer rectifiable current can be written as $T+kQ$, with $T, Q \in \Rect_d(M)$ and $T$ is a representative mod $k$. In particular, $\Mass^k([T])= \Mass(T)$ and $\spt^k([T])=\spt(T)$, whenever $T \in \Rect_d(M)$ and $T$ is a representative mod $k$. If $T \in \Rect_d^k(M)$ and $B \subset \mathbb{R}^N$ is a Borel subset, we define $T\res B:= [T\res B]$, where $T$ is a rectifiable representative of $T$ mod $k$.

The boundary operator $\partial^k$ on $\mathbf{F}^k_d(M)$ is defined by $\partial^k [T]=[\partial T]$, since if $T=S$ mod $k$, then also $\partial T=\partial S$ mod $k$ for any flat chains $T, S \in \mathbf{F}_d(M)$. A flat chain $T \in \mathbf{F}_d(M)$ is a cycle mod $k$ if $\partial T=0$ mod $k$, \emph{i.e.} $\partial^k [T]=0$ and it is a boundary mod $k$ if there exists $S \in \mathbf{F}_{d+1}(M)$ such that $T=\partial S$ mod $k$, \emph{i.e.} $[T]=\partial^k [S]$; the spaces of $d$-dimensional mod $k$ integral cycles and $d$-dimensional mod $k$ integral boundaries with $\spt^k(\cdot) \subset M$ are denoted by $\mathcal{Z}_d(M,\mathbb{Z}_k)$ and $\mathcal{B}_d(M,\mathbb{Z}_k)$, respectively; note that $\mathcal{B}_d(M,\mathbb{Z}_k)$ is a normal subgroup of $\mathcal{Z}_d(M,\mathbb{Z}_k)$. If $M$ is taken to be a closed smooth oriented manifold $\Mcal$, we define its mod $k$ homology groups as 
\begin{equation}\label{e:mod-k-homol}
    H_*(\mathcal{M},\mathbb{Z}_k)= \mathcal{Z}_*(\mathcal{M},\mathbb{Z}_k)/\mathcal{B}_*(\Mcal,\mathbb{Z}_k)\,,
\end{equation}
and it is possible to prove, as a consequence of the mod $k$ deformation theorem, that the homology theory so defined satisfies all seven Eilenberg-Steenrod Axioms, and hence it is equivalent to the singular homology of $\mathcal{M}$ with coefficients in $\mathbb{Z}_k$. In the following sections, we will slightly abuse notation between currents mod $k$ and their representatives. Moreover, when the coefficients are not specified, homology and cohomology have to be understood with coefficients in $\mathbb{Z}_k$.

\subsection{Stable homotopy theory and Borel-Moore homology}\label{s:spectra}

We refer the reader to \cite{ABCD} for an introduction to the main tools in homotopy theory and Thom's theory of oriented cobordism; instead, we will introduce here a few basic notion from stable homotopy theory and spectra. We will limit ourselves to the basic definitions, with no hope to be self-contained, and focus instead on the useful computational properties of such objects that are relevant in the next sections. We will mostly follow \cite{Rudyak} and refer the reader to \cite{Adams, Switzer} for other general references on these topics.

For pointed spaces $X$ and $Y$, we denote by $[X, Y]$ the set of pointed homotopy classes of pointed maps $X \rightarrow Y$. We will often write $Z_+$ for the pointed space associated to a space $Z$. We recall that the \emph{smash product} $X \wedge Y$ is defined by $$X \wedge Y:=X \times Y / X \vee Y, $$ where $X \vee Y$ is viewed as the subspace of $X \times Y$ consisting of those pairs $(x, y)$ such that either $x$ is the basepoint of $X$ or $y$ is the basepoint of $Y$. Intuitively, the smash product operation can be understood as the appropriate analogue of taking the Cartesian product in the category of pointed spaces and it enjoys useful properties (on compactly generated spaces): it is functorial in both variables, the two-point space serves as unit, it is associative and commutative. The appropriate pointed analogue of the function space is the subspace $F(X, Y)$ of $Y^X$ consisting of the pointed maps, with the constant pointed map as basepoint. 

Viewing $S^1$ (which we always consider as a pointed space with basepoint 1) as $I / \partial I$ for a unit interval $I$, we define the (reduced) \emph{suspension} of a pointed space $X$ (with basepoint $*$) to be $\Sigma X:=X \wedge S^1$, that is, $$\Sigma X=X \times S^1 /\left(\{*\} \times S^1 \cup X \times\{1\}\right).$$

Dually, we define the path space of a pointed space $X$ to be $P X:=F(I, X)$, where $I$ is given the basepoint 0. Thus, the points of $P X$ are the paths in $X$ that start at the basepoint. We define the loop space of $X$ to be $\Omega X:=F(S^1, X)$. Its points are the loops at the basepoint. Note that we have the adjunction $F(\Sigma X, Y) \cong F(X, \Omega Y).$ Passing to the set of path components $\pi_0$ (and by the identification of $[X, Y]$ with $\pi_0(F(X, Y))$), this gives that $$[\Sigma X, Y] \cong[X, \Omega Y] .$$

Recall that $S^m \wedge S^n = S^{m+n} $, so that by induction we have that the $n$-fold suspension is thus $$\Sigma^nX= S^1 \wedge \Sigma^{n-1}X = S^1 \wedge (S^{n-1} \wedge X) = (S^1 \wedge S^{n-1}) \wedge X=S^{n} \wedge X.$$ Analogously one has the iterated loop space $\Omega^{n}X= F(S^n, X).$ 

We recall the following fundamental theorem, \emph{cf.} \cite[Theorem 6.26]{Switzer}. For a based space $X$, we define the suspension homomorphism $\Sigma_*: \pi_i(X) \rightarrow \pi_{i+1}(\Sigma X)$ which sends $[f] \mapsto [\Sigma f]$, where $$\Sigma f:=f \wedge \operatorname{id}_{S^1}: S^{i+1} = S^i \wedge S^1 \rightarrow X \wedge S^1=\Sigma X.$$

\begin{thm}[Freudenthal Suspension]\label{t:Freudenthal}
Assume $X$ is an $n$-connected $CW$-complex (based at a $0$-cell), with $n\geq 0$. Then the suspension homomorphism $\Sigma_*: \pi_i(X) \rightarrow \pi_{i+1}(\Sigma X)$ is an isomorphism if $1\leq i \leq 2n$ and a surjection if $i=2n+1$. In other words, the map $\eta: X \rightarrow \Omega \Sigma X$ is a $2n+1$-equivalence.
\end{thm}

{\bf The category of spectra}

\begin{definition}
A \emph{spectrum}\footnote{This definition is equivalent to the following one: $E$ is a sequence $\{E_n, \sigma_n\}_{n \in \Z}$ of pointed $CW$-complexes $E_n$ and pointed $CW$-embeddings $\sigma_n: \Sigma E_n \rightarrow E_{n+1}$ (that is, $\sigma_n(\Sigma E_n)$ is a subcomplex of $E_{n+1}$, \emph{cf.} \cite[Lemma-Definition II.1.19]{Rudyak}).} $\{E_n,\sigma_n\}_{n \in \N}$ is a sequence $\{E_n\}_{n}$ of pointed spaces along with pointed maps $\sigma_n:\Sigma E_n \rightarrow E_{n+1}$. 
\end{definition}

Note that the maps can be taken to be $\sigma'_n:  E_n \rightarrow \Omega E_{n+1}, $ since $\Sigma$ and $\Omega$ are adjoints. When the maps are clear from context, we will simply write $E=\{E_n\}$ to denote a spectrum.

Given a spectrum $E$ and an integer $N$, we define the spectrum $\Sigma^NE$ by setting $(\Sigma^NE)_n= E_{n+N}$ with maps $\sigma_{n+N}:\Sigma(\Sigma^NE)_n \rightarrow (\Sigma^NE)_{n+1}$. Moreover, for every pointed $CW$-complex $X$ it is possible to define the spectrum $\Sigma^{\infty}X$ as $(\Sigma^{\infty}X)_n:=\Sigma^nX$, with maps the identity, \emph{i.e.} $\sigma_n: \Sigma (\Sigma^nX) \rightarrow \Sigma^{n+1}X.$ For example, the spectrum $\Sigma^{\infty}S^0$ is the sphere spectrum $\{S^n, i_n\},$ where $i_n: \Sigma S^n \rightarrow S^{n+1}$. 

Let $E=\{E_n, \sigma_n\}$ and $F=\{F_n, t_n\}$ two spectra. A map of spectra $f: E\rightarrow F$ is a family of pointed maps of spaces $f_n: E_n\rightarrow F_n$ of spaces such that $f_{n+1}\circ \sigma_n = t_n \circ \Sigma f_n.$ We skip the definition of morphism of spectra and only recall that any function of spectra is a morphism of spectra, \emph{cf.} \cite[II.1.3]{Rudyak}.

For every pointed space $X$ and every spectrum $E=\{E_n, \sigma_n\}$ it is possible to define the spectrum $X \wedge E:=\{X\wedge E_n, 1 \wedge \sigma_n\}$. In particular, we can define the suspension $S^1 \wedge E$ of a spectrum $E$.

A \emph{homotopy} between two maps $g_0,g_1:E \rightarrow F$ is a map $G: E \wedge [0,1]_{+} \rightarrow F$ restricting to $g_0$ over $\{0\}_{+}$ and to $g_1$ over $\{1\}_{+}$; it is possible to extend this definition to morphisms of spectra, see \cite[Definition II.1.9]{Rudyak}, and with this definition it is possible to define a suitable notion of \emph{homotopy equivalence} of spectra $E\simeq F$, analogous to that of spaces. We denote by $[E,F]$ the set of homotopy classes of morphisms $E\rightarrow F$. 

Working with spectra one usually cares about the limiting behavior of the spaces $E_n$ as $n \rightarrow \infty$. To capture this idea, one defines the homotopy groups of a spectrum as a colimit (direct limit) of the homotopy groups of the spaces $E_n$. Because of the suspensions, the homotopy groups shift in degree as we go from each level to the next.

Let $S=\Sigma^{\infty}S^0$ be the sphere spectrum and $E$ any spectrum. For each $d \in \mathbb{Z}$, the $d^{th}$ \emph{(stable) homotopy group} $\pi_d(E)$ of $E$ is defined as the group $[\Sigma^dS, E]$. It is possible to see that $\pi_d(E)= \lim_{n\rightarrow \infty}\pi_{d+n}(E_n)$, where the colimit is that of the direct system of abelian groups $$\cdots \longrightarrow \pi_{d+n}(E_n) \xrightarrow{s}  \pi_{d+n+1}(\Sigma E_n) \xrightarrow{(\sigma_n)_*} \pi_{d+n+1}(E_{n+1}) \longrightarrow \cdots,$$ where each homotopy group is taken at the basepoint. The operation $s$ suspends each map $S^{d+n} \rightarrow E_n$ to produce a map $S^{d+n+1} \rightarrow \Sigma E_n$. When $d$ is negative, the system is only defined for $n \geq|d|$. Given a morphism of spectra $\varphi: E \rightarrow F$, there is a well-defined induced homomorphism $\varphi_{*}: \pi_d(E) \rightarrow \pi_d(F)$. Note that $\pi_d(E)$ can be non-zero even for $d<0$, a simple example being the spectrum $E=\Sigma^{-N}S$. Analogously one can define the notion of \emph{homology} of a spectrum $E$ as $H_d(E)=\lim_{n\rightarrow \infty}\widetilde{H}_{d+n}(E_n)$.

An analog of the Whitehead theorem is valid for spectra, see \cite[Theorem II.1.10]{Rudyak}: a morphism $\varphi: E \rightarrow F$ is an equivalence if and only if the induced homomorphism $\varphi_{*}: \pi_d(E) \rightarrow \pi_d(F)$ is an isomorphism for every integer $d$.

Important examples of spectra are \emph{Thom spectra}. Let $MSO(n)$ be the Thom space of the universal oriented $n$-plane bundle $\gamma^n$ over the infinite oriented Grassmannian of $n$-planes $BSO(n)$. Then, the Whitney sum $\gamma^n \oplus \varepsilon$ is classified by the map $BSO(n) \rightarrow BSO(n+1)$ and admits a bundle map to $\gamma^{n+1}$, where $\varepsilon$ is the trivial one-dimensional vector bundle. The Thom space of $\gamma^n \oplus \varepsilon$ is $\Sigma MSO(n)$ and the Thom space of $\gamma^{n+1}$ is $MSO(n+1)$, and hence for each $n\geq 1$ we get a map $$\Sigma MSO(n) \rightarrow MSO(n+1).$$ The \emph{Thom spectrum} $MSO$ is the spectrum in which the $n^{th}$-space is $MSO(n)$ and the maps are $\sigma_n: \Sigma MSO(n) \rightarrow MSO(n+1)$. 

A spectrum is an $\Omega$-\emph{spectrum} if for every $n$ the map $\sigma'_n:  E_n \rightarrow \Omega E_{n+1}$ is a homotopy equivalence. When $E$ is an $\Omega$-spectrum, the map from each group $\pi_{d+n}(E_n)$ to the colimit $\pi_d(E)$ is always an isomorphism.

Let $G$ be an abelian group. Recall that an Eilenberg-MacLane space $K(G, n)$ is a $CW$-complex whose $n^{th}$-homotopy group is identified with $G$ and all other homotopy groups are zero. We define an \emph{Eilenberg-MacLane spectrum} $HG$ by $(HG)_{n}:=K(G, n)$, and maps the canonical equivalences $K(G, n) \simeq \Omega K(G, n+1).$ By construction, this is an $\Omega$-spectrum with only one stable homotopy group $\pi_d(HG) \simeq G$, for $d=0$, and $\pi_d(HG) \simeq 0$, for $d \neq0$.

We remark that the sphere spectrum $S$ and Thom spectra are not $\Omega$-spectra, but we have the following important result of Adams, see \cite[Proposition II.1.21]{Rudyak}.

\begin{thm}
Every spectrum $E=\{E_n, \sigma_n\}$ is homotopy equivalent to some $\Omega$-spectrum.
\end{thm}

The $0^{th}$-space $E_0$ of an $\Omega$-spectrum $E$ is also called an \emph{infinite loop space} and denoted $\Omega^{\infty}E$, see \emph{cf.} \cite[Proposition-Definition II.3.27]{Rudyak} for some of its properties. In particular, if $F$ is a spectrum and $E$ is an $\Omega$-spectrum equivalent to it, then $\Omega^{\infty}F \simeq E_0$. We define the \emph{free infinite loop space} of a pointed space $X$ as $QX:=\Omega^\infty\Sigma^\infty X=\lim_{i\rightarrow \infty}\Omega^i\Sigma^i X,$
where the colimit is taken along the maps $\Omega^i\Sigma^iX\to\Omega^{i+1}\Sigma^{i+1}X$ obtained by applying $\Omega^i$ to the suspension map $\Sigma^iX\to\Omega\Sigma^{i+1}X$.

There is a construction which assigns to spectra $E,F$ a certain spectrum called the \emph{smash product} $E\wedge F$ of spectra. One can introduce $E\wedge F$ as generalization of the smash product $E\wedge X$ of a spectrum and a space. We refer to \cite{Adams} for the exact construction of the smash product of spectra and we recall here some properties that will be useful for us in the next sections.

The smash product $E\wedge F$ of spectra is a covariant functor of each of its arguments; denoting by $S$ the sphere spectrum, there are the following natural equivalences:
\begin{itemize}
\item $(E \wedge F) \wedge G \simeq E \wedge(F \wedge G), $
\item $E \wedge F \simeq F \wedge E,$ 
\item $S \wedge E \simeq E ,$
\item $\Sigma E \wedge F \simeq \Sigma(E \wedge F) .$
\end{itemize}
If $f: E \rightarrow F$ is an equivalence then the morphism $f \wedge id_G: E \wedge G \rightarrow F \wedge G$ is.

A spectrum $E$ is called $n$-\emph{connected} if $\pi_i(E)=0$ for $i \leq n$. A spectrum is called \emph{connected} if it is $-1$-connected. A morphism $\varphi: E \rightarrow F$ of spectra is called $n$-\emph{connected}, or an $n$-\emph{equivalence}, if $\varphi_*: \pi_i(E) \rightarrow \pi_i(F)$ is an isomorphism for $i < n$ and an epimorphism for $i=n$.

\begin{pro}[\protect{\cite[Proposition II.4.5]{Rudyak}}]\label{p:equivalences}
Let $E,F,G$ be spectra. The following hold.
\begin{itemize}
\item[\textit{(i)}] If $E$ is $m$-connected and $F$ is $n$-connected, then $E \wedge F$ is $m+n+1$-connected.
\item[\textit{(ii)}]  If $E$ is $m$-connected and $\varphi: F \rightarrow G$ is an $n$-equivalence, then $$id_E \wedge \varphi : E \wedge F \rightarrow E \wedge G$$ is an $m+n+1$-equivalence.
\item[\textit{(iii)}]  Given integers $N, k$, let $f: E \rightarrow F$ be a map of spectra such that $f_n: E_n \rightarrow F_n$ is an $n+k$-equivalence for every $n>N$. Then $f$ is a $k$-equivalence.
\end{itemize}
\end{pro}

A way to change coefficients in homology theories is to take smash products with Moore spaces. We say that a $CW$ complex $X$ is a \emph{Moore space} if $H_n(X,\Z_k) \simeq \Z_k$ and $\widetilde{H}_i(X)=0$ for all $i\neq n$. We denote it as $M(\Z_k,n)$. They are constructed starting with $S^n$ as $n$-skeleton and attaching an $n+1$-cell by a map of degree $k$. For example, when $n=1$ and $k=2$ we obtain $\mathbb{R}\mathbb{P}^2$. In particular, there exists a spectrum $M\Z_k$ defined uniquely up to equivalence by the following properties: $\pi_i(M\Z_k)=0$ for $i<0$; $\pi_0(M\Z_k)\simeq \Z_k \simeq H_0(M\Z_k)$; $H_i(M\Z_k)=0$ for $i\neq0.$ Such a spectrum is called \emph{Moore spectrum} of $\Z_k$.

Let $X$ be a space. We recall that there is a strong connection between spectra and (co)homology theories see \cite[II.3.22]{Rudyak}. Every spectrum $E=\{E_n\}_{n\in\N}$ yields a (generalized) homology theory $E_*(-)$ and cohomology theory $E^*(-)$ by the formulae $$E_i(X):= \lim_{n\rightarrow \infty}\pi_{i+n}(X \wedge E_n), \quad E^i(X):= \lim_{n\rightarrow \infty}\left[\Sigma^n X, E_{i+n}\right].$$ In particular, they are dual to each other; more precisely we have the following, \emph{cf.} \cite[Propositions II.3.23, II.3.25]{Rudyak}.

\begin{pro}\label{p:poincareduality}
For every spectrum $E$, the homology theory $E_*$ is dual to the cohomology theory $E^*$.
\end{pro}
 
Finally, we have the following representability property. 
\begin{pro}\label{p:srepresentability}
If $E=\{E_n\}_{n\in\N}$ is an $\Omega$-spectrum, then for every space $X$ there is a natural equivalence $$\widetilde{E}^n(X)\simeq \left[X, E_n\right] \simeq \left[X, \Omega^{\infty}\Sigma^nE\right].$$
\end{pro}

{\bf Borel-Moore homology}

The Borel-Moore homology $H_*^{BM}(M,\Z)$ of a space $M$ is a notion of homology for locally compact spaces like noncompact manifolds introduced by Borel and Moore in \cite{BorelMoore} from a sheaf-theoretic viewpoint. Geometrically it can be thought of as constructed from possibly infinite singular chains with locally finite supports (hence, it is also known as the homology with closed supports or locally finite homology). If $M$ is compact, $H_*^{BM}(M,\Z)$ is isomorphic to the ordinary singular homology $H_*(M,\Z)$. A useful property is that an $m$-dimensional submanifold $\Sigma$ in a manifold $M$ determines a fundamental class $[\Sigma] \in H_m^{BM}(M,\Z)$, even when $\Sigma$ is not compact\footnote{For example, if $\Sigma$ admits a triangulation, then its fundamental class is represented by the sum of all the top dimensional simplices.}. Moreover, if $M$ is an oriented $m+n$-dimensional manifold, then Poincaré duality holds: $H_m^{BM}(M,\Z)$ is Poincaré dual to the singular cohomology $H^{n}(M,\Z)$ and a Thom representability criterion holds for noncompact manifolds using Borel-Moore homology. Let us recall a few basic properties in the case of $\Z$-coefficients, but the notion is much more general and is valid for $\Z_k$-coefficients and fundamental classes of $\Z_k$-submanifolds; we refer to \cite[Chapter 19]{Fulton} for more details (see also \cite[Chapter 2]{Cattalani}).

Let $C_i^{BM}(M)$ be the abelian group of (infinite) sums $$u=\sum_\sigma a_\sigma \sigma,$$ where $\sigma$ runs over the set of all continuous maps from the standard $i$-simplex $\Delta^i$ to $X$ and each $a_\sigma$ is an integer, such that for each compact subset $K$ of $M$, we have $a_\sigma \neq 0$ for only finitely many $\sigma$ whose image meets $K$. Then the usual definition of the boundary $\partial$ of a singular chain makes these abelian groups into a chain complex: $$ \cdots \rightarrow C_2^{BM}(X) \rightarrow C_1^{BM}(X) \rightarrow C_0^{BM}(X) \rightarrow 0 .$$

The Borel-Moore homology groups $H_i^{BM}(M)$ are the homology groups of this chain complex, that is

$$
H_i^{BM}(X)=\operatorname{ker}(\partial: C_i^{BM}(X) \rightarrow C_{i-1}^{BM}(X)) / \operatorname{im}(\partial: C_{i+1}^{BM}(X) \rightarrow C_i^{BM}(X)) .
$$

If $M$ is compact, then every locally finite chain is in fact finite and hence Borel-Moore homology $H_i^{BM}(X)$ coincides with the ordinary singular homology $H_i(M)$. Borel-Moore homology is a covariant functor with respect to proper maps. That is, a proper map $f: M \rightarrow N$ induces a pushforward homomorphism $H_i^{BM}(M) \rightarrow H_i^{BM}(N)$ for all integers $i$. Moreover, it is a contravariant functor with respect to inclusions of open subsets. That is, for $U$ open in $M$, there is a restriction homomorphism $H_i^{BM}(M) \rightarrow H_i^{BM}(U)$.

For any locally compact space $M$ and any closed subset $K$, with $U=M \backslash K$ the complement, there is a long exact sequence, \emph{cf.} \cite[19.1 (6)]{Fulton}:

\begin{equation}\label{e:les}
\cdots \rightarrow H_i^{BM}(K) \rightarrow H_i^{BM}(M) \rightarrow H_i^{BM}(U) \rightarrow H_{i-1}^{BM}(K) \rightarrow \cdots
\end{equation}

Poincaré duality extends to noncompact manifolds using Borel-Moore homology. Namely, for an oriented $m+n$-dimensional manifold $M$, Poincaré duality is an isomorphism from singular cohomology to Borel-Moore homology,
$$
H^i(X) \simeq H_{n-i}^{B M}(X),
$$
for all integers $i$. 

Borel-Moore homology is usually employed to define fundamental classes for arbitrary (possibly singular) spaces, see \cite[Lemma 19.1.1]{Fulton}: we will employ it to work with $m$-dimensional homology classes of $M \setminus K$, where $M$ is a smooth closed oriented manifold of dimension $m+n$ and $K$ will either denote the $m-2$ or the $m-5$-dimensional skeleton of a smooth triangulation $\mathcal{K}$ of $M$. We remark that $K$ has codimension at least 2 in the $m$-dimensional class and, by the long exact sequence in \eqref{e:les}, the $m$-dimensional homology classes of $M$ and $M\setminus K$ are isomorphic.

Finally, we recall a Thom criterion valid for noncompact manifolds, see \cite[Théorème II.1']{Thom54} for an even more general version (we can simply choose the family $(\Phi)$ of support sets therein to be the one of all closed subsets).

\begin{thm}
Let $\mathcal{M}$ be an oriented smooth manifold without boundary, possibly noncompact, and let $\tau \in H^{BM}_{m}(\mathcal{M},\mathbb Z)$ an integral homology class. Then $\tau$ is represented by a properly embedded oriented smooth submanifold $i:\Sigma \rightarrow \mathcal{M}$ with $i_*[\Sigma]_{BM}=\tau$ if and only if there exists a map $f :\mathcal{M} \rightarrow MSO(n)$ pulling back the Thom class to the Poincaré dual of $\tau$, \emph{i.e.} $f^*(u)=x$, where $x \in H^{n}(\mathcal{M},\Z)$.
\end{thm}

\subsection{$\Z_k$-smooth manifolds}\label{s:Zkmanifolds}

We now recall some basic notions of the theory of $\Z_k$-manifolds introduced by Morgan and Sullivan \cite{SullivanHaup, MorSul, MITnotes}. Such objects can be defined in the category of smooth, PL, or topological manifolds: we always assume them to be $C^{\infty}$-smooth.

\begin{definition}
A closed $m$-dimensional $\Z_k$-\emph{manifold} (or $\Z_k$-\emph{smooth manifold}) is given by the triple $\overline{M}=(M,\delta M,\theta_i)$ where:  
\begin{itemize}
\item[\textit{i)}] $M$ is a compact oriented smooth $m$-manifold, with boundary $\partial M$;
\item[\textit{ii)}] $\delta M$ is a compact oriented smooth $m-1$-manifold without boundary, called the \emph{Bockstein};
\item[\textit{iii)}] $\theta_i:\delta M \rightarrow \partial M$, for $i\in\{1,...,k\}$, are $k$ disjoint embeddings such that we have an orientation-preserving diffeomorphism $\partial M=\bigsqcup_{i\in\{1,...,k\}}\theta_i(\delta M)$.
\end{itemize}
\end{definition}

There is an associated quotient space $M_k$ given by the identification on $M$ of the $k$ copies of $\delta M$ via the embeddings $\theta_i$'s. With a slight abuse of notation we will often refer to the latter quotient (or to the couple $(M, \delta M)$) when dealing with $\Z_k$-manifolds. We will always assume without further mention that the sheets are labeled $\{1,...,k\}$ near the Bockstein $\delta M$.

\begin{example}[$\mathbb{Z}_k$-spheres]\label{e:Zksphere}
Let $M$ be the $n$-sphere $S^n$ with $k$ disjoint open disks removed. Then $M$ identifies with a $\mathbb{Z}_k$-manifold with $\delta M=S^{n-1}$, which we call the $\mathbb{Z}_k-n$-sphere and we denote it by $S_k^n$.
\end{example}

Note that in particular any closed oriented manifold is a (nonsingular) $\Z_k$-manifold with empty Bockstein $\delta M$. 

\begin{remark}\label{r:Z2}
It is important to remark that if $(M,\delta M)$ is $\Z_2$-manifold with nonempty Bockstein then the quotient $M_2$ has the structure of a non-orientable manifold, but not every non-orientable manifold is a $\Z_2$-manifold. In particular, a $\Z_2$-manifold $M_2$ has the additional geometric structure that $M_2$ is oriented away from $\delta M$, and the orientation flips across it, so that the first Stiefel-Whitney class $w_1(M_2)$ is the Poincaré dual of $\delta M$ and hence the mod 2 reduction of an integral cohomology class\footnote{This class of objects is also known as \emph{Wall manifolds}, \emph{i.e.} manifolds for which the first Stiefel-Whitney class is the mod 2 reduction of an integral class.}. This is reflected in the different notions of classifying spaces for embeddings of such manifolds. For general unoriented manifolds it is the Thom space of the universal $n$-plane bundle, while, for $\Z_2$-manifolds, it is the (corresponding finite level space of the) smash product $MSO \wedge M\Z_2$, which provides more structure on the normal bundles. For instance, a Klein bottle is an example of $\mathbb{Z}_2$-manifold, while $\mathbb{R}\mathbb{P}^2$ is instead a nonorientable manifold that cannot be given the structure of a $\Z_2$-manifold, since $w_1(\mathbb{R}\mathbb{P}^2)\neq 0$ but $H^1(\mathbb{R}\mathbb{P}^2,\Z)=0$.
\end{remark}

It is possible to define the notion of a $\Z_k$-manifolds with boundary. 

\begin{definition}
An $m$-dimensional $\Z_k$-\emph{manifold with boundary} is given by the triple $\overline{M}=(M,\delta M,\psi_i)$ where:
\begin{itemize}
\item[\textit{i)}] $M$ is a compact $m$-dimensional manifold, with boundary $\partial M$;
\item[\textit{ii)}] $\delta M$ is a compact $m-1$-dimensional manifold (the Bockstein) \emph{with boundary} $\partial\delta M$;
\item[\textit{iii)}] $\psi_i:\delta M \rightarrow \partial M $, with $i\in \{1,...,k\}$, are $k$ disjoint embeddings such that the triple 
$$ \left(\partial M-\text{int}\big( \bigsqcup_{i\in\{1,...,k\}} \psi_i(\delta M)\big),\partial \delta M,{\psi_i}|_{\partial \delta M}\right)$$ defines a closed $m-1$-dimensional $\Z_k$-manifold $(N,\delta N,\theta_i)$. 
\end{itemize}
\end{definition}
The closed $m-1$-dimensional $\Z_k$-manifold $(N,\delta N,\theta_i)$ is called the $\Z_k$-\emph{boundary} of the $\Z_k$-manifold with boundary $\overline{M}$, and is denoted by $\partial \overline{M}=(N,\delta N,\theta_i)$. Again, there is the quotient space $M_k$ resulting after the identification on $M$ of the $k$ embedded copies of $\delta M$ via the $\psi_i$'s.

From this description one can see that closed $\Z_k$-manifolds have fundamental homology classes with $\Z_k$-coefficients and $\Z_k$-manifolds with boundary have a relative fundamental class. Thus they serve as geometric representatives of $\Z_k$-homology in any space. 

A map between $\Z_k$-manifolds is a map between the respective quotient spaces, which sends the singularity (i.e. the identified images $\theta_i(\delta M)$) to the singularity, and preserves the cone-on-$k$ points structure (with the labeling) near the singularity. Thus embeddings have normal bundles, and we have a theory of transversality, see \cite[Chapter 1]{MorSul}. Moreover, it is possible to define a notion of $\Z_k$-\emph{maps} from $\Z_k$-manifolds to smooth manifolds, passing through the quotient $\Z_k$-manifold; we refer to \cite{Oregon} for the rigorous definition of them.

Indeed, we give a geometric definition of the tangent bundle of a $\Z_k$-smooth manifold $(M,\delta M)$. Embed $\delta M$ into some $\R^N$ for some large $N$. Embed the cone over $k$ points $C(k)$ in $\R^2$ so that all lines are tangent at the cone point in a cusp-like manner (see \cite[page 474]{MorSul}). The product of these two embeddings, gives one of $\delta M \times C(k) \rightarrow \R^{N+2}$. Its boundary is $\delta M \times k=\partial \overline{M}$. Extend this embedding to one of $M$. Together these embeddings define one of the quotient $M_k \rightarrow \R^{N+2}$. The tangent bundle is the set of tangent planes with the natural topology and it is well-defined since all planes over the Bockstein agree. We denote this bundle by $\tau_M$. Taking perpendicular planes gives the normal bundle $\nu_M$. We remark that it is possible to alternatively define the tangent bundle just by taking the one of $M$, denoted $\tau_M$, and use $d\varphi_{i,j}$ to identify $\tau_M|_{\delta M_i}$ with $\tau_M|_{\delta M_j}$, $\varphi_{i,j}$ being the isomorphisms between the $i^{th}$ and $j^{th}$ copy of $\partial M$.

An $m$-dimensional \emph{singular $\Z_k$-manifold} in a space $X$ is a closed oriented smooth $m$-dimensional $\Z_k$-manifold $\overline{M}=(M,\delta M,\theta_i)$ together with a continuous map $f:M\rightarrow X$ such that $f\circ\theta_i=f\circ\theta_j$ for $i,j\in \Z_k$. The $\Z_k$-bordism group $\Omega_m(X, \Z_k)$ is given by the equivalence classes of $m$-dimensional singular $\Z_k$-manifolds $(\overline{M},f)$ under the clear $\Z_k$-bordism relation; that is, a singular $\Z_k$-manifold $(M , \delta M, \theta_i, f : M \rightarrow X)$ represents a class in $\Omega_m(X, \Z_k)$.

In \cite{BRS} the authors showed that every generalized homology theory is a (co)bordism theory of a suitable class of manifolds \emph{with singularities}. In particular, for a spectrum $E$, we call its corresponding homology theory the (co)bordism theory of manifolds with $X$-singularity; moreover, if $E,F$ are spectra, then $E\wedge F$-singularity combines $E$-singularity and $F$-singularity. For instance, $MSO \wedge M\Z_k$ corresponds to the (co)bordism theory of (oriented) $\Z_k$-smooth manifolds; analogously, the Moore spectrum $M\Z_k$ for $\Z_k$ corresponds to the (co)bordism theory of framed $\Z_k$-manifolds, which is easy to see since the Thom spectrum representing framed (co)bordism is the sphere spectrum $S$.

We will need to consider embeddings of a $\Z_k$-manifold into a smooth manifold. For what follows, let $\Sigma$ be a $\Z_k$-manifold and let $M$ be a smooth manifold. A $\Z_k$-map $f: \Sigma \rightarrow M$ is said to be a $\Z_k$-\emph{immersion} if it is an immersion when considering $\Sigma$ as a smooth manifold with boundary. Additionally, a $\Z_k$-immersion $f: \Sigma \rightarrow M$ is said to be a $\Z_k$-\emph{embedding} if the induced map on the quotient $\Z_k$-manifold $\Sigma_k$ is an embedding.

A suitable variant of Thom's criterion applies also to the mod $k$ realizability problem.

\begin{thm}\label{t:criterionmodk}
Let $\mathcal{M}$ and $\tau$ be as in Assumption \ref{a:1}, and $n>m+1$. Then a mod $k$ homology class $\tau \in H_m(\mathcal{M},\mathbb{Z}_k)$ is realizable by an embedded $\mathbb{Z}_k$-smooth manifold $N$ if and only if there exists a map $f:\mathcal{M} \rightarrow \Omega^{\infty}\Sigma^n(MSO \wedge M\Z_k)$ such that it pulls back the $\mathbb{Z}_k$ Thom class to the Poincaré dual of $\tau$. 
\end{thm}

\begin{remark}\label{r:important}
Before coming to the proof of Theorem \ref{t:criterionmodk} we comment on the target space $\Omega^{\infty}\Sigma^n(MSO \wedge M\Z_k)$, which is the $n^{th}$-space of the $\Omega$-spectrum equivalent to $MSO \wedge M\Z_k$. No transversality can be applied directly from this space, since it is not the Thom of some vector bundle (unless after sufficiently many suspensions), but the stability condition $n\geq m+1$ allows us to exploit the $2n-1$-equivalence of the map $g: MSO(n-1) \wedge M(\Z_k,1) \rightarrow \Omega^{\infty}\Sigma^n(MSO \wedge M\Z_k)$ to be always able to find a lift from the latter space to the former, see Lemma \ref{l:MSOtoOmega} below. This $2n-1$-equivalence is also the reason why in dealing with $\mathcal{M}\setminus \mathcal{K}^{m-5}$, which is equivalent to a complex of dimension $n+4$, we can obtain this lift every time $n\geq 5$.

Indeed, note that assuming $MSO(n-1) \wedge M(\Z_k,1)$ is the classifying space for embedded $\Z_k$-smooth submanifold $(N, \delta N)$ imposes the restriction on the normal bundle of each sheet $N$ in $\mathcal{M}$ to be of the split form $\nu_N=\xi^{n-1}_N \oplus \varepsilon$. Near $\partial N$, instead, the rank $n+1$ normal bundle in $\mathcal{M}$ has the structure of $\nu_{\partial N}=\xi^{n-1}_{\partial N} \oplus \varepsilon^2$, from the compatibility condition of the sheets along their boundary $\nu_{N_i|_{\partial N}} \oplus \varepsilon = \nu_{\partial N}$. We have not yet been able to identify which is the correct classifying space for embedded submanifolds without this restriction on the normal bundle, but we remark that, in stable regimes, this is not an issue since $MSO(n-1) \wedge M(\Z_k,1)$ retains the spectrum $MSO \wedge M\Z_k$ after stabilization; in other words, every normal bundle of $N$ can be turned into the structure imposed above.

We remark that, \emph{a posteriori}, we could also have studied the map from $MSO(n-1) \wedge M(\Z_k,1)$ to $K(\Z_k,n)$ representing the mod $k$ Thom class directly, aiming to prove that it is an $n+4$-equivalence for all $n\geq 5$. This would have required a refined study of the cohomology of such spaces with coefficients in $\Z_p$, for $p$ prime, which is arguably more involved than our stable homotopy theory computations in Proposition \ref{p:4-equivalence} and Lemma \ref{l:MSOtoOmega}.
\end{remark}

\begin{proof}[Proof of Theorem \ref{t:criterionmodk}]
Let us denote $T_n:= \Omega^{\infty}\Sigma^n(MSO \wedge M\Z_k)$ and assume that we have a map $f: \mathcal{M} \rightarrow T_n$ pulling back the mod $k$ Thom class $u_k \in H^n(T_n,\Z_k)$ to the $x$ Poincaré dual of $\tau$. Since Lemma \ref{l:MSOtoOmega} below guarantees that the map $g: MSO(n-1) \wedge M(\Z_k,1) \rightarrow T_n$ is a $2n-1$ equivalence and the dimension of $\mathcal{M}$ is $m+n$, then by our stability assumption $n\geq m+1$ we obtain a lift from $T_n$ to $MSO(n-1) \wedge M(\Z_k,1)$ such that the diagram commutes up to homotopy. That is, we obtain a map (not relabeled) $$f : \mathcal{M} \rightarrow MSO(n-1) \wedge M(\Z_k,1)$$ pulling back the mod $k$ Thom class $\tilde{u}_k \in H^n(MSO(n-1) \wedge M(\Z_k,1),\Z_k)$ to $x$.

Now, we exploit a two-stage standard transversality procedure described in Morgan-Sullivan \cite{MorSul}: this is done first obtaining the Bockstein $\delta N$ and then then regular sheets attached to it, yielding the $\Z_k$-smooth submanifold $(N,\delta N)$. Let us start recalling that $M(\Z_k,1)$ is obtained as $S^1 \cup_k D^2$ by attaching a 2-cell onto $S^1$ by a degree $k$ map $h$. Denote by $\hat{e}_i$ the center of the $i$-cell $e_i$ of $M(\Z_k,1), i=0,1,2$. Let $X_k$ denote the union of $k$ closed radial segments in $D^2$ joining $\hat{e}_2$ to $h^{-1}(\hat{e}_1)$. 

Let us first make $f$ smooth and transverse to the zero section $$BSO(n-1) \times \{\hat{e}_2\} \in MSO(n-1) \wedge M(\Z_k,1),$$ so that we obtain $\delta N= f^{-1}(BSO(n-1) \times \{\hat{e}_2\})$, which is a smooth embedded oriented submanifold of $\mathcal{M}$ of codimension $n+1$, since we have to count the contribution given by the codimension 2 of the point $\hat{e}_2$ with respect to the 2-cell. Its normal bundle in $\mathcal{M}$ has the structure $\nu_{\delta N}= \xi^{n-1}_{\delta N} \oplus \varepsilon^2$. Additionally, denoting $D(\cdot)$ the disk bundle, we note that $\delta N$ has an open neighborhood $U$ diffeomorphic to $D(\xi^{n-2}_{\delta N}) \times e_2$. This allows us to homotope $f$ to a map (not relabeled) that coincides with $f$ in $U$ and collapses everything outside $U$ onto $MSO(n-1) \wedge S^1$, which is the Thom space of the universal oriented $n$-plane bundle with one added line $\gamma^{n-1} \oplus \varepsilon$.

We now construct the $m$-dimensional sheets attaching to the neighborhood of the Bockstein $\delta N$. Denote $\mathcal{M}_0:=\mathcal{M}\setminus \text{int} U$. By the homotopy extension property, we can perturb $f$ keeping it fixed inside $U$ and so that $f_0:\mathcal{M}_0 \rightarrow MSO(n-1) \wedge S^1$ is smooth and trasverse to $BSO(n-1) \times \{\hat{e}_1\}$. Hence, we obtain $N_0:=f_0^{-1}(BSO(n-1) \times \{\hat{e}_1\})$, which is a compact smooth oriented submanifold of codimension $n$, with normal bundle $\nu_{N_0}= \xi^{n-1}_{N_0} \oplus \varepsilon$ in $\mathcal{M}_0$ and boundary $\partial N_0$ in $\partial U$.

Denote $X_k-g^{-1}(\hat{e}_1)$ by $\stackrel{\circ}{X}_k$. Then, we finally define $$N:= N_0 \cup \text{cl} \big(f^{-1}(BSO(n-1) \times \stackrel{\circ}{X}_k)\big).$$ Then $(N,\delta N)$ is an embedded $\Z_k$-smooth manifold such that, near $\delta N$, $N$ is diffeomorphic to $\delta N \times C(k)$. Moreover, $\nu_{N\setminus \delta N} = \xi_{N\setminus \delta N}^{n-1} \oplus \varepsilon$ and $\nu_{\delta N} = \xi_{\delta N}^{n-1} \oplus \varepsilon^2$ in $\mathcal{M}$. It is also clear that the mod $k$ fundamental class $[N]$ represents $\tau$, \emph{i.e.} $f_{*}[N]=\tau$.

Conversely, assume that $\tau \in H_m(\mathcal{M},\mathbb{Z}_k)$ is realizable by an embedded $\mathbb{Z}_k$-smooth manifold $N$, that is there exists a smooth $\Z_k$-embedding $i: N \rightarrow \mathcal{M}$ such that $i_*[N]=\tau$. In particular, the map $i$ gives us a map from $N$ to the $\Z_k$-cobordism group $\Omega_{m}(\mathcal{M},\Z_k)$ and, as shown in \cite{MorSul}, there is a natural map $\Omega_{*}(\mathcal{M},\Z_k) \rightarrow (MSO\wedge M\Z_k)_{*} (\mathcal{M}),$ where the latter is the generalized homology theory represented by the spectrum $MSO\wedge M\Z_k$. Finally, by the Poincaré duality in Proposition \ref{p:poincareduality}, this is isomorphic to the generalized cohomology theory generated by $MSO\wedge M\Z_k$, and hence together with Proposition \ref{p:srepresentability} we can write $$(MSO\wedge M\Z_k)_{m} (\mathcal{M}) \simeq (MSO\wedge M\Z_k)^{m} (\mathcal{M}) \simeq [\mathcal{M}, \Omega^{\infty}\Sigma^{n}(MSO\wedge M\Z_k)],$$ hence providing a map $f: \mathcal{M} \rightarrow \Omega^{\infty}\Sigma^{n}(MSO\wedge M\Z_k)$ (representative of that cohomology class) pulling back the mod $k$ Thom class to the Poincaré dual of $\tau$, by commutativity of the diagram.
\end{proof}

\begin{remark}
An adaptation of this criterion holds for mod $k$ Borel-Moore homology and properly embedded $\Z_k$-manifolds representing a mod $k$ homology class, but we skip the details.
\end{remark}

We conclude this section with the following lemma, which was used in the proof of Theorem \ref{t:criterionmodk} to connect $\Omega^{\infty}\Sigma^n(MSO \wedge M\Z_k)$ with an ``unstable'' model for $\Z_k$-manifolds.

\begin{lem}\label{l:MSOtoOmega}
 The map $g: MSO(n-1) \wedge M(\Z_k,1) \rightarrow \Omega^{\infty}\Sigma^n(MSO \wedge M\Z_k)$ is a $2n-1$-equivalence. 
 \end{lem}

\begin{proof}
Denote $T^n_k:=MSO(n-1) \wedge M(\Z_k,1)$. We will prove that both maps $$f_1: T^n_k \rightarrow \Omega^{\infty}\Sigma^{\infty} T^n_k, \quad f_2: \Omega^{\infty}\Sigma^{\infty} T^n_k \rightarrow \Omega^{\infty}\Sigma^{n} MSO \wedge M\Z_k$$ are $2n-1$-equivalences, hence giving us the result. 

We first note that $T^n_k$ is $n-1$-connected, as it is the smash product of an $n-2$-connected space and a connected one. By the Freudenthal suspension theorem, \emph{cf.} Theorem \ref{t:Freudenthal}, we know that the map from $T^n_k$ to $\Omega\Sigma T^n_k$ is a $2n-1$-equivalence; we also know that $\Sigma^i T^n_k$ is $(n-1+i)$-connected. Hence the maps $\Omega^i\Sigma^i T^n_k\to\Omega^{i+1}\Sigma^{i+1} T^n_k$ are $(2n-1+i)$-equivalences, so that the stabilization map $f_1: T^n_k \rightarrow \Omega^{\infty}\Sigma^{\infty} T^n_k$ is a $2n-1$-equivalence.

Moreover, note that the map $h:\Sigma^{\infty}MSO(n-1) \rightarrow \Sigma^{n-1}MSO$ of Thom spectra is a $2(n-1)$-equivalence. This follows from the fact that at the level of classifying spaces the map $BSO(n-1) \rightarrow BSO(n-1 +i)$ is an $n-1$-equivalence for every $i\geq1$. Hence, at the level of Thom spaces, we obtain that the map $\Sigma^{i}MSO(n-1) \to MSO(n-1+i)$ is a $(2(n-1) +i)$-equivalence. Passing to suspension spectra and desuspending by $i$ we obtain that the map $\Sigma^{\infty}MSO(n-1) \to \Sigma^{-i}\Sigma^{\infty}MSO(n-1+i)$ is a $2(n-1)$-equivalence. Furthermore, note that $$\Sigma^{-i}\Sigma^{\infty}MSO(n-1+i)= \Sigma^{n-1}\Sigma^{-(n-1+i)}\Sigma^{\infty}MSO(n-1+i)$$ so that, after stabilization, we obtain that the map $h$ is a $2(n-1)$-equivalence as well. 

Now, since the suspension spectrum $\Sigma^{\infty}M(\Z_k,1)$ is equivalent to the one-shifted Moore spectrum $\Sigma M\Z_k$, we have that the map $$H: \Sigma^{\infty}T^n_k \simeq \Sigma^{\infty}MSO(n-1) \wedge  \Sigma^{\infty}M(\Z_k,1) \rightarrow \Sigma^{n-1}MSO \wedge \Sigma M\Z_k \simeq \Sigma^{n} MSO \wedge M\Z_k$$ is a $2n-1$-equivalence; this follows from the $-1$-connectedness of $M\Z_k$ and the fact that suspending a map of sectra raises its connectivity by one. Finally, applying the infinite loop space $\Omega^{\infty}$ preserves the $2n-1$-equivalence of the map, so that the map $f_2: \Omega^{\infty}\Sigma^{\infty} T^n_k \rightarrow \Omega^{\infty}\Sigma^{n} MSO \wedge M\Z_k$ is a $2n-1$-equivalence, obtaining our conclusion.
\end{proof}

\section{Geometric measure theory propositions}\label{s:GMT}

We now state the mod $k$ analogs of the key intermediate geometric measure theory results needed for the proof of Theorem \ref{t:1}. They have independent interest and, in particular, the $\Z_k$-smooth approximation away from the codimension 2 skeleton, established in Theorem \ref{p:poly_approx_prescribedsing} below, requires particular attention.

\begin{remark}
We remark that using Borel-Moore homology classes allows us to deal with homology classes of $\mathcal{M}\setminus \mathcal{K}^{m-2}$ directly, instead of removing a suitable smooth neighborhood of the codimension 2 skeleton $\mathcal{K}^{m-2}$ of any smooth triangulation $\mathcal{K}$ of $\mathcal{M}$. This allows us to avoid a large number of tedious technicalities compared to \cite{ABCD, nonorientato} where one needs to build a relative homology class representative; instead we directly obtain a map from $\mathcal{M}\setminus \mathcal{K}^{m-2}$ to the suitable Thom space pulling back the mod $k$ Thom class to the Poincaré dual of $\tau$.
\end{remark}

\subsection{Codimension 2 $\mathbb{Z}_k$-smooth approximation}

The first step towards proving Theorem \ref{t:1} is the following mod $k$ version of the codimension 2 smooth approximation theorem in \cite[Section 4]{ABCD}, that is an upgraded version of the Federer and Fleming strong polyhedral approximation theorem in the spirit of \cite[$(4.2.21)^{\nu}$]{Federerbook}\footnote{See also \cite[Theorem 3.4]{Marchesestuvard}.}, where the approximands are smooth submanifolds outside of the $(m-2)$-skeleton $\mathcal{K}^{m-2}$ of a smooth triangulation of $\mathcal{M}$.

\begin{pro} \label{p:poly_approx_prescribedsing}
Let $\mathcal{M}$ be as in Assumption \ref{a:1} and let $T$ be an $m$-dimensional integral cycle mod $k$ in $\mathcal{M}$. 
For every fixed $\varepsilon_c > 0$ there is a mod $k$ integral cycle $\tilde{P}$ homologous to $T$ and a smooth triangulation $\mathcal{K}$ of $\mathcal{M}$ with the following properties:
\begin{itemize}
\item[($a_0$)] $\mathbb{M}(\tilde{P}) \leq (1+\varepsilon_c) \mathbb{M}(T)$;
\item[($b_0$)] $\mathbb{F} (T-\tilde{P})\leq \varepsilon_c$;
\item[($c_0$)] $\spt(\tilde{P}) \subset \{ x : \dist(x,\spt(T))\leq \varepsilon_c\}$;
\item[($d_0$)] for every $p\in \spt (\tilde{P}) \setminus \mathcal{K}^{m-1}$ there is a neighborhood $U$ of $p$ and a smooth $m$-dimensional oriented submanifold $\Gamma$ of $\mathcal{M} \cap U$ with boundary contained in $\mathcal{M} \cap \partial U$ and such that $\tilde{P} \res U = \llbracket \Gamma \rrbracket$. 
\end{itemize}
\vspace{0.25cm}
\noindent
Furthermore, for a sufficiently small $\delta_c'>0$ and any $\delta_c<\delta_c'$ we can find another mod $k$ integral cycle $P'$ homologous to $P$ with the following properties:
\vspace{0.25cm}
\begin{itemize}
\item[($a_1$)] $\mathbb{M} (P') \leq (1+2\varepsilon_c) \mathbb{M} (T)$ and $\mathbb{F} (T-P') \leq 2\varepsilon_c$;
\item[($b_1$)] $\spt(P') \subset \{ x : \dist(x,\spt(T))\leq 2\varepsilon_c\}$;
\item[($c_1$)] $\|P'\| (B_{\delta'_c} (\mathcal{K}^{m-2})) \leq 3\varepsilon_c$;
\item[($d_1$)] $P'\res \mathcal{M}\setminus \mathcal{K}^{m-2}= \llbracket\Gamma \rrbracket\modk$ for some embedded $\mathbb{Z}_k$-smooth submanifold $\Gamma$ of $\mathcal{M}\setminus \mathcal{K}^{m-2}$ without boundary in $\mathcal{M}\setminus \mathcal{K}^{m-2}$.
\end{itemize}
\end{pro}

The heart of the proof of Proposition \ref{p:poly_approx_prescribedsing} is similar to the one in the previous settings: to develop a suitable perturbation of an approximating polyhedral cycle which is a $\Z_k$-smooth embedded submanifold outside of the codimension 2 skeleton $\mathcal{K}^{m-2}$ of the smooth triangulation $\mathcal{K}$ of $\mathcal{M}$; this is then applied to obtain a map from $\mathcal{M}\setminus \mathcal{K}^{m-2}$ to the suitable Thom space pulling back the mod $k$ Thom class to the restriction of Poincaré dual of $\tau$. 

Particular care is needed in the following steps: 
\begin{itemize}
    \item multiplicity removal from $m$-dimensional faces coming potentially with positive multiplicity (analogously to that in the integral setting);
    \item splitting any multiples of $(m-1)$-dimensional faces along which $m$-dimensional faces meet in pairs with opposite orientations and smooth these out (cross-section-wise) as in \cite{ABCD}, so that one remains with $\ell \cdot k$ sheets, for $\ell \geq2$, all with the same sign;
    \item splitting these $\ell \cdot k$ sheets into $\ell$ groups of $k$ sheets that are all with the same orientation and structure along the $(m-1)$-dimensional faces.
\end{itemize}
Finally we obtain a mod $k$ polyhedral cycle with multiplicity 1 in each face, close in the flat topology, in support and with mass close to that of $T$. Moreover, each component is either smooth oriented or has a codimension one stratum of singularities where exactly $k$ sheets join together; such singular components together form the (possibly disconnected) Bockstein. We hence obtain a mod $k$ integral cycle which is a $\Z_k$-smooth manifold outside of $\mathcal{K}^{m-2}$ for a smooth triangulation $\mathcal{K}$ of $\mathcal{M}$.

\begin{proof}
Choose a mod $k$ integral cycle $T$ in $\mathcal{M}$ and consider it as a mod $k$ integral cycle in $\mathbb{R}^N$ for a suitable choice of $N$. Apply now the \emph{first approximation} of \cite[Proposition 4.1]{ABCD}, which can be readily adapted to mod $k$ integral cycles, to obtain a mod $k$ integral cycle $P$ homologous to $T$ with $\spt(P) \subset \mathcal{M}$ and such that:
\begin{itemize}
\item $\mathbb{M}(P) \leq (1+\varepsilon_c/2) \mathbb{M}(T)$;
\item $\mathbb{F} (T-P)\leq \varepsilon_c/2$;
\item $\spt(P) \subset \{ x : \dist(x,\spt(T))\leq \varepsilon_c/2\}$;
\item $P = \sum_{F\in \mathcal{F}^m} a_F \llbracket F \rrbracket$ mod $k$, $a_F \in \Z_k$ and where $\mathcal{F}^m$ is the collection of $m$-dimensional oriented cells of $\mathcal{K}$, for a suitable choice of smooth triangulation $\mathcal{K}=\mathcal{K}(T, \varepsilon_c)$.
\end{itemize}

\noindent
Starting with $P$ we now perform the \emph{second approximation} of \cite[Proposition 4.1]{ABCD}, which was divided into Step 1 (regularization in the complement of $\mathcal{K}^{m-1}$) and Step 2 (removal of the $(m-1)$-dimensional singularities). This time, we need not, and cannot generally hope to, remove all of the $(m-1)$-dimensional singularities without creating any additional boundary.

{\bf Regularization on $\mathcal{M} \setminus \mathcal{K}^{m-1}$}

We start by fixing, for the above triangulation $\mathcal{K}$, a suitable polyhedral submanifold $K$ of $\mathbb R^N$ and a piecewise smooth homeomorphism $\psi: K\to \mathcal{M}$ which realizes the triangulation $\mathcal{K}$ in the sense that, for a suitable triangulation $\mathcal{T}$ of the polyhedron $K$, the following holds: for every cell $F$ of $\mathcal{K}$, its diffeomorphic preimage $\psi^{-1} (F)$ is a simplex of $\mathcal{T}$. The mod $k$ current $P$ is then given by
\[
P = \sum_{F\in \mathcal{F}^m} \beta_F \llbracket F \rrbracket\, ,
\]
where $\beta_F\in \Z_k$ and $\mathcal{F}^m$ is the collection of $m$-dimensional cells of $\mathcal{K}$. After choosing integer representatives of least (absolute) value and orienting the faces accordingly, we write $P=\sum_{F\in\mathcal F^m}\beta_F\llbracket F\rrbracket$ mod $k$, for $\beta_F\in\{0,\ldots,\lfloor k/2\rfloor\}.$ Since the relative interiors of distinct faces are disjoint, the mod $k$ mass satisfies $\mathbb M(P)=
\sum_{F\in\mathcal F^m}\beta_F\,\mathcal H^m(F).$ For every cell $F$ with $\beta_F>0$ we consider $\Delta:= \psi^{-1} (F)$ and we let $\Gamma$ be an $(m+1)$-dimensional simplex of the triangulation $\mathcal{T}$ which contains $\Delta$. We will replace $\beta_F \llbracket F \rrbracket$ with $\sum_{j=1}^{\beta_F} \llbracket \psi (\Delta_j)\rrbracket$, where the $\Delta_j\subset \Gamma$ are diffeomorphic images of $\Delta \subset \Gamma$ with $\partial \Delta_j=\partial \Delta$ and $\Delta_{j'}\cap \Delta_{j''}=\partial \Delta$ for every $j'\neq j''$. In order to define the $\Delta_j$ we will use the following elementary lemma, which is proven in \cite[Lemma 4.7]{ABCD}.

\begin{lem}\label{l:costruisco_funzione}
Consider the $m$-dimensional simplex $\Omega\subset \mathbb R^{m}$ which is the convex hull of $\{e_0, e_1, \ldots, e_m\}$, where $e_0=0$ and $e_1, \ldots, e_m$ is the standard basis. For each $i\in \{0, 1, \ldots, m\}$ let $F_i$ be the relative interior of the $(m-1)$-dimensional face of $\Omega$ spanned by $e_0, \ldots, e_{i-1}, e_{i+1}, \ldots, e_m$. In other words, $F_i$ consists of those points $p$ which can be written as convex combinations $\sum_j \lambda_j e_j$ with $\lambda_i=0$ and $\lambda_j>0$ for every $j\neq i$.

Then there is a Lipschitz function $f:\Omega \to \mathbb R$ and a neighborhood $V$ of $\bigcup_i F_i$ such that
\begin{itemize}
\item[(a)] $f$ is positive and smooth in the interior of $\Omega$;
\item[(b)] $f (x) = \dist (x, \partial \Omega)$ for every $x\in V$;
\item[(c)] ${\rm Lip}\, (f) \leq C$ for some constant $C=C(m)$.
\end{itemize}
\end{lem}

With Lemma \ref{l:costruisco_funzione} at hand, we are ready to define $\Delta_j$. First of all let $\{v_0, v_1, \ldots, v_m\}$ be the extremal points of $\Delta$, $\pi$ the $m$-dimensional linear space spanned by $\{v_1-v_0, \ldots , v_m-v_0\}$ and then let $v_{m+1}$ be the only unit vector orthogonal to $\pi$ with the property that $\frac{1}{m+1} \sum_i v_i + \gamma v_{m+1} \in \Gamma$ for every $\gamma$ sufficiently small. Let $A: \Delta \to \Omega$ be the affine map defined by $A (v_i)=e_i$. Choose then positive numbers $0<\delta_1 < \delta_2 < \ldots < \delta_{\beta_F}$ and define $\Delta_j$ as 
\[
\Delta_j = \{x+ \delta_j\, f(A(x)) v_{m+1} : x\in \Delta\}\, .
\]
Provided the $\delta_{\beta_F}$ is chosen sufficiently small, each $\Delta_j$ is contained in $\Gamma$.
Note moreover that, by construction, $\psi (\Delta_{j'})\cap \psi (\Delta_{j''})\subset \mathcal{K}^{m-1}$, the $(m-1)$-dimensional skeleton of $\mathcal{K}$.

We perform the above construction for all $m$-dimensional cells $F$ with $\beta_F>0$. Upon enumerating them as $\{F^i\}_{i\in I}$, for each $i \in I$ we denote by $\{\Delta^i_j\}_{j\in J_i}$ the corresponding $m$-dimensional cells in the polyhedron $K$, by $\delta^i$ the number $\delta_{\beta_{F_i}}$ and by $\{F^i_j\}_{j\in J_i}$ their images through $\psi$. We note further that, by choosing the $\delta_{\beta_{F_i}}$ sufficiently small, we can ensure that $F^i_j \cap F^{i'}_{j'} \subset \mathcal{K}^{m-1}$ for every choice of distinct pairs $(i,j)$ and $(i',j')$.

Consider now the mod $k$ integral current
\[
\tilde{P} := \sum_{i\in I} \sum_{j \in J_i} \llbracket F^i_j\rrbracket\, .
\]
We have $\partial (\sum_j \llbracket F^i_j\rrbracket) -\partial (\beta^i \llbracket F_i\rrbracket)=0$ mod $k$ and so $\sum_j \llbracket F^i_j\rrbracket - \beta^i \llbracket F_i \rrbracket$ is a mod $k$ integral cycle $T_i$. Moreover, via the homotopy formula we can ensure that the mass of the filling is as small as required, provided that $\delta^i$ is chosen small enough. We can further ensure that $\mathcal{H}^m (F^i_j)$ is as close as required to $\mathcal{H}^m (F)$ via the area formula. In particular, we conclude that, upon choosing the $\delta^i$'s sufficiently small,
\begin{itemize}
\item $\tilde{P}$ is in the same homology class as $P$ (in $\mathcal{M}$);
\item $\spt (\tilde{P}) \subset B_{\varepsilon_c} (\spt (T))$;
\item $\mathbb{M} (\tilde{P}) \leq (1+\varepsilon_c) \mathbb{M} (T)$;
\item $\mathbb{F} (T-\tilde{P}) \leq \varepsilon_c$;
\item $\spt (\tilde{P}) \setminus \mathcal{K}^{m-1}$ is smooth and is taken with multiplicity $1$ by $\tilde{P}$;
\end{itemize}
More precisely, for every $p\in \spt (\tilde{P}) \setminus \mathcal{K}^{m-1}$ there is a neighborhood $U$ of $p$ and a smooth oriented $m$-dimensional submanifold $\Gamma$ of $\mathcal{M} \cap U$ with boundary contained in $\mathcal{M} \cap \partial U$ such that 
$\tilde{P} \res U = \llbracket \Gamma \rrbracket$. This yields the desired claims $(a_0)$-$(d_0)$.

{\bf Splitting from $\Z_{\ell \cdot k+ 2a}$ to $\Z_k$}

In the last step of the proof of Proposition \ref{p:poly_approx_prescribedsing} we will slightly perturb $\tilde{P}$ by removing its additional $(m-1)$-dimensional singularities for $m$-dimensional faces joining the $m-1$-faces with opposite orientation. This is exactly the same as in the oriented setting, in particular Step 2 in the second approximation in proof of \cite[Proposition 4.1]{ABCD}. The remaining $\ell \cdot k$ $m$-faces all have the same sign, and we split them into $\ell$ groups by a slight perturbation, which topologically corresponds to the map $r: \mathbb{Z}_{\ell \cdot k} \rightarrow \mathbb{Z}_{k}$ of \cite[page 472]{MorSul}.

To do so, consider an arbitrary face $F^l$ as in the previous subsection and let $\sigma_i$ be an arbitrary $(m-1)$-dimensional face of $F^l$. Fix $\delta_c'>0$, to be determined, and $\delta_c\in (0,\delta_c')$. The goal is to modify $\tilde{P}$ in a neighborhood of $\sigma_i\setminus B_{\delta_c} (\mathcal{K}^{m-2})$ to a new mod $k$ current $P'$ in the same homology class, close to it in terms of mass, support and flat norm, and with the property that $P'$ is a $\mathbb{Z}_k$-smooth manifold in that neighborhood. The neighborhood in which we will perturb $\tilde{P}$ is of the form $B_{\lambda_i} (\sigma_i)\setminus B_{\delta_c} (\mathcal{K}^{m-2})$. First of all, given the structure of $\tilde{P}$ obtained in the previous subsection, if $\lambda_i$ is sufficiently small, there is an open subset $\Delta_i\subset \mathbb R^{m-1}$ and a smooth parametrization 
\[
\Phi : \Delta_i \times B_{\lambda_i}^{n+1} \to \mathcal{M}
\]
of a normal neighborhood $\mathcal{N}_i$ of $\sigma_i \setminus B_{\delta_c} (\mathcal{K}^{m-2})$ with thickness $\lambda_i$ and with the property that $\spt (\tilde{P}\res \mathcal{N}_i)$ can be described in the following way. There are a finite number $\ell \cdot k$ of distinct unit vectors $w_1, \ldots , w_{\ell \cdot k} \in \partial B_1^{n+1} \subset \mathbb R^{n+1}$ such that, if we let 
\[
\Lambda_q = \{(x, s w_q) : x \in \Delta_i, 0 < s < \lambda_i\}\, ,
\]
then $\tilde{P} \res \mathcal{N}_i = \sum_{q =1}^{\ell \cdot k} \epsilon_{q} \Phi_\sharp \llbracket \Lambda_q \rrbracket$, where $\epsilon_q$ takes values in $\{-1,1\}$; note that we know the number of unit vectors is a multiple of $k$ since $\tilde{P}$ has no boundary mod $k$ in $\mathcal{N}_i$. This additionally tells us that for exactly $\ell_1 \cdot k$ of the $\ell\cdot k$ terms in the sum, the numbers $\epsilon_q$ equal $1$, while the remaining $\ell_2\cdot k$ equal $-1$,  with $\ell_1 + \ell_2=\ell$. 

We partition the $(m-1)$-dimensional faces from $\{\Lambda_q\}_{q=1}^{\ell\cdot k}$ into $\ell$ groups $ G_1,\dots,G_\ell$ such that $|G_j|=k$ and so that the first $\ell_1$ groups contain only positively-oriented faces and the remaining $\ell_2$ groups contain only negatively-oriented faces. Note that the normal space to $\sigma_i$ has dimension $n+1$ and since $n\geq2$, it has always dimension at least 3. In a sufficiently small normal slice through an interior point of $\sigma_i$, each $m$-dimensional face appears as a ray in direction $w_q$, for a total of $\ell \cdot k$ rays. For each $j=1,\dots,\ell$, label the independent normal directions $w_q$ corresponding to the faces in $G_j$ by $w_{a,j}$, $a=1,\dots,k$. Now choose distinct points $p_1,\dots,p_\ell \in \partial B_{1/2}^{n+1}$, and replace each line segment $[0,w_{a,j}]$ with the ones given by $[p_j,w_{a,j}]$. Since $\dim \R^{n+1}\geq 3$, a generic choice of the $p_j$'s make the segments $[p_j,w_{a,j}]$ disjoint, except at their prescribed common endpoints $w_{a,j} \in 	\partial B_1^{n+1}$, where one possibly creates a singularity. However, we can easily get rid of these singularities by smooth embedded arcs that agree with the original radial segment near each $w_{a,j}$; call each smooth arc $\gamma_{a,j} \subset\overline B_1^{n+1}$, for $a=1,...,k$ and $j=1,...,\ell$, joining $p_j$ to $w_{a,j}$, which agrees with the original radial segment near $w_{a,j}$. We may arrange that the arcs remain straight near $p_j$ and are mutually disjoint apart from the common endpoint $p_j$ for arcs belonging to the same group.

Finally, we let $\tau_t: \mathbb R^{n+1} \to \mathbb R^{n+1}$ be the homothety $y\mapsto t y$. We are now ready to define a replacement for $\tilde{P}\res \mathcal{N}_i$. For $0\leq t<\lambda_i$, define

$$ \Gamma^i_{a,\,j,t} := \tau_t(\gamma_{a,\,j}) \cup \{r \, w_{a,\,j}:t\le r\leq\lambda_i\}. $$

In other words, $\Gamma^i_{a,\,j,t}$ consists of the scaled connecting arc together with the unchanged radial portion; by construction these two sets join smoothly for $t\in (0,\lambda_i)$. We then fix a smooth compactly supported function $\psi_i:\mathbb R^{m-1} \to [0,\infty)$ which is positive on $\Delta_i$ and vanishes on $\partial \Delta_i$, a small positive number $\kappa_i$ (to be determined), and we define the sheets $$ S^i_{a,\,j} := \left\{ (x,y):x\in\Delta_i,\; y\in\Gamma^i_{a,\,j,\kappa_i\psi_i(x)} \right\}. $$

With orientations chosen consistently with the original sheets, we can now define 

$$\Sigma^i:= \sum_{j=1}^{\ell}\epsilon_j \sum_{a=1}^{k}\llbracket S^i_{a,\,j}\rrbracket,$$ with $\epsilon_j$ equal to +1 for $j\le\ell_1$, and $-1$ for $\ell_1<j\le\ell$.

Finally we can define $$P^i:=\Phi_\sharp\Sigma^i,$$ which satisfies $\partial P^i = \partial (\tilde{P}\res \mathcal{N}_i)$ mod $k$. Moreover we can make $\mathbb{F} (P^i - \tilde{P}\res \mathcal{N}_i)$ and $\mathbb M (P^i) - \mathbb M (\tilde{P}\res \mathcal{N}_i)$ smaller than any desired threshold by choosing $\kappa_i$ sufficiently small. Finally, note that $\Sigma_i$ is $\mathbb{Z}_k$-smooth in $\mathcal{N}_i$ since, along each new joining locus $$Y_j^i:= \{(x,\kappa_i\psi_i(x)p_j):x \in \Delta_i\},$$ exactly $k$ $m$-dimensional faces of the same sign join, therefore yielding the local $\mathbb{Z}_k$-smooth manifold model.

We next enumerate all the $(m-1)$-dimensional faces $\sigma_i$ of all the $m$-dimensional faces $F^l$ as $\sigma_1, \sigma_2, \ldots , \sigma_N$. We choose our parameters in such a way that the sets $\mathcal{N}_i$ are pairwise disjoint. Our final current $P'$ is then defined as
\[
P' := \sum_i P^i + \tilde{P} \res \mathcal{M} \setminus \bigcup_i \mathcal{N}_i\, .
\]
By construction, $P'\res \mathcal{M}\setminus B_{\delta_c}(\mathcal{K}^{m-2})$ is $\Z_k$-smooth. Moreover, by choosing the parameters $\delta_c',\lambda_i, \kappa_i$ accordingly depending on $\eps_c$ and combining with the bounds $(a_0)$-$(c_0)$ that we have already established, $P'$ is homologous to $\tilde P$. We also achieve the estimates $(a_1)$ and $(b_1)$, since we can make $\spt (P')$ arbitrarily close to $\spt (\tilde{P})$, $\mathbb{M} (P')$ arbitrarily close to $\mathbb{M} (\tilde{P})$, and $\mathbb{F} (P'-\tilde{P})$ arbitrarily small. Finally, the estimate $(c_1)$ on $\|P'\| (B_{\delta_c'} (\mathcal{K}^{m-2}))$ is established analogously to that in \cite{ABCD}. More precisely, first note that we have
\begin{align*}
&\|P'\| (\mathcal{M})\leq \|P\| (\mathcal{M}) + 2 \varepsilon_c\,;\\
&\|P'\| \big(\mathcal{M}\setminus B_{\delta_c'} (\mathcal{K}^{m-2})\big)\geq\|P\| \big(\mathcal{M}\setminus B_{\delta_c'} (\mathcal{K}^{m-2})\big)\, .
\end{align*}
Hence,
\begin{align*}
\|P'\| \big(B_{\delta'_c}(\mathcal{K}^{m-2})\big) &= \|P'\| (\mathcal{M}) - \|P'\| \big( \mathcal{M}\setminus B_{\delta_c'} (\mathcal{K}^{m-2})\big)\\
&\leq \|P\| (\mathcal{M}) + 2\varepsilon_c - \|P\| \big(\mathcal{M}\setminus B_{\delta_c'} (\mathcal{K}^{m-2})\big) \\
&= \|P\| \big(B_{\delta_c'} (\mathcal{K}^{m-2})\big) + 2\varepsilon_c
\end{align*}
for every fixed choice of $\delta_c\in (0,\delta_c')$. So, as long as $\delta_c'$ is additionally chosen small enough to ensure that $\|P\| (B_{\delta_c'} (\mathcal{K}^{m-2})) \leq \varepsilon_c$, we deduce the desired estimate.
We finally note that the same arguments can be easily improved to ensure that the mod $k$ integral cycle $P'$ can be chosen so that its singularities are all contained in $\mathcal{K}^{m-2}$.
\end{proof}

\begin{remark}
Note that $P'$ is an embedded $\Z_k$-smooth manifold $(M,\delta M)$ outside of $\mathcal{K}^{m-2}$ with trivial normal bundle $\nu(M) = \varepsilon^{n-1} \oplus \varepsilon^1$ and $\nu(\delta M)=\varepsilon^{n-1} \oplus \varepsilon^2$ in $\mathcal{M}\setminus \mathcal{K}^{m-2}$, where the $\varepsilon^1$ direction is prescribed near each binding as the infinitesimal rotation direction of the ``sheets''. This gives rise to a map $$\tilde{f}: \mathcal{M}\setminus \mathcal{K}^{m-2} \rightarrow S^{n-1} \wedge M(\Z_k,1) \simeq M(\Z_k,n) $$ with values in the Moore space $M(\Z_k,n)$, which is the classifying space for framed embedded $\Z_k$-smooth submanifolds, pulling back the fundamental class to the restriction of the Poincaré dual of $\tau$ to $\mathcal{M}\setminus \mathcal{K}^{m-2}$. Note also that, from the standard inclusion of the $(n-1)$-sphere into $MSO(n-1)$, we have $S^{n-1} \wedge M(\Z_k,1) \xrightarrow{i \wedge id} MSO(n-1) \wedge M(\Z_k,1)$ and hence obtain a map $$f: \mathcal{M}\setminus \mathcal{K}^{m-2} \rightarrow MSO(n-1) \wedge M(\Z_k,1)$$ pulling back the fundamental class to the restriction of the Poincaré dual of $\tau$ to $\mathcal{M}\setminus \mathcal{K}^{m-2}.$
\end{remark}

\subsection{Squeezing lemma}

The next ingredient we need for the proof of Theorem \ref{t:1} is the following mod $k$ version of the \emph{squeezing} lemma in \cite[Section 4]{ABCD}. Here, we are given two $m$-dimensional mod $k$ integral cycles $S$ and $R$ which agree outside of a sufficiently small neighborhood of the $(m-2)$-dimensional skeleton $\mathcal{K}^{m-2}$. It is then possible to show that $S$ and $R$ represent the same homology class; there is a smooth deformation $R'$ of $R$ which is close, in terms of mass and in flat norm, to $S$ and $R'$ coincides with $S$ outside a slightly larger neighborhood of $\mathcal{K}^{m-2}$.

\begin{pro}\label{p:squash}
Let $m$ and $\mathcal{M}$ be as in Assumption \ref{a:1} and $\mathcal{K}$ a smooth triangulation of $\mathcal{M}$. Then for every $\varepsilon_d >0$ and every $\eta_d>0$ there exists $\delta_d(\varepsilon_d, \eta_d, \mathcal{K}, \mathcal{M})>0$ with the following property. Suppose that $S$ and $R$ are $m$-dimensional mod $k$ integral cycles in $\mathcal{M}$ and that 
\[
S\res\mathcal{M}\setminus B_{\delta_d}(\mathcal{K}^{m-2})=R\res\mathcal{M}\setminus B_{\delta_d}(\mathcal{K}^{m-2})\, .
\]
Then $S$ and $R$ are homologous. Moreover, there exists a mod $k$ integral cycle $R'$ in the same homology class and a diffeomorphism $\Phi$ of $\mathcal{M}$ with the following properties:
\begin{enumerate}\itemsep0.2em
\item $\mathbb{M}(R') \leq (1+\varepsilon_d) \mathbb{M}(S)$;
\item $\mathbb{F}(R'-S)\leq C \big(\varepsilon_d \mathbb{M} (S)+2\|S\| (B_{\eta_d} (\mathcal{K}^{m-2}))\big)^{\frac{m+1}{m}}$, with $C=C (\mathcal{M})$;
\item $R'\res\mathcal{M}\setminus B_{\eta_d}(\mathcal{K}^{m-2})=S\res\mathcal{M}\setminus B_{\eta_d}(\mathcal{K}^{m-2})$;
\item $\Phi$ is in the isotopy class of the identity and $R'= \Phi_\sharp R$.
\end{enumerate}
\end{pro}

The proof of Proposition \ref{p:squash} follows \emph{verbatim} that of \cite[Proposition 4.3]{ABCD} (and has already been recalled in the mod 2 version in \cite[Proposition 4.3]{nonorientato}); we hence state the proposition without rewriting the proof.

\section{Proof of Theorem \ref{t:1}}\label{s:proofofmain}

The final key ingredient for the the proof of our main theorem is the following.

\begin{pro}\label{p:4-equivalence}
The map $h: MSO \wedge M\mathbb{Z}_k \rightarrow H\mathbb{Z}_k$ representing the mod $k$ Thom class is a $4$-equivalence of spectra.
\end{pro}

\begin{proof}
Recall from \cite[Lemma 5.2]{ABCD} that the map $H : MSO(n) \rightarrow K(\Z,n)$ representing the Thom class $u \in H^{n}(MSO(n),\Z)$ is an $n+4$-equivalence for any $n \geq 1$. By $(iii)$ in Proposition \ref{p:equivalences}, the map of spectra $$F: MSO \rightarrow H\Z$$ is a 4-equivalence. Since the Moore spectrum is connected and $\varphi: H\Z \wedge M\Z_k \rightarrow H\Z_k$ is an equivalence, we conclude by $(ii)$ in Proposition \ref{p:equivalences} that the map $h\simeq \varphi \circ (F\wedge id)$ representing the mod $k$ Thom class $$MSO \wedge M\Z_k \xrightarrow{F\wedge id} H\Z \wedge M\Z_k \xrightarrow{\varphi} H\Z_k$$ is a $4$-equivalence of spectra.
\end{proof}

We are now ready to prove the main theorem of this article.

\begin{proof}[Proof of Theorem \ref{t:1}]

Fix $ \varepsilon_c > 0$, whose choice will be specified later, and a mod $k$ integral cycle $T$ in a homology class $\tau\in H_m (\mathcal{M}, \mathbb Z_k)$. First of all we apply Proposition \ref{p:poly_approx_prescribedsing} to find a sufficiently small $\delta'_c$, a suitable triangulation $\mathcal{K}$ of $\mathcal{M}$ and a new mod $k$ integral cycle $P'=:S$ with the property that $S$ is in the same homology class of $T$ and the following facts hold:
\begin{itemize}
\item $\mathbb{M} (S) \leq \mathbb{M} (T) + 3\varepsilon_c$ and $\mathbb{F} (S-T) < 3 \varepsilon_c$;
\item $\|S\| (B_{\delta_c'} (\mathcal{K}^{m-2})) \leq 3\varepsilon_c$;
\item $B_{\delta_c'} (\mathcal{K}^{m-2})$ is homotopy equivalent to $\mathcal{K}^{m-2}$;
\item $S\res \mathcal{M} \setminus \mathcal{K}^{m-2}= \llbracket \Gamma \rrbracket$ for a smooth submanifold $\Gamma$.
\end{itemize}

Now define $Q_2 := \mathcal{M} \setminus \mathcal{K}^{m-2}$.

The current $\llbracket \Gamma \rrbracket$ obtained from Proposition \ref{p:poly_approx_prescribedsing} is (when restricted to $Q_2$ and not relabeled) a smooth properly embedded $\Z_k$-submanifold of $Q_2$. Denoting by $x \in H^{n}(\mathcal{M})$ the Poincaré dual of $\tau$, note that its restriction $x_{|Q_2} \in H^n(Q_2)$ to $Q_2$ is the Poincaré dual of a Borel-Moore homology class which is represented by the smooth properly embedded submanifold $\Gamma \subset Q_2$. By the (locally finite version of) the Thom collapse construction of Theorem \ref{t:criterionmodk} applied to $\Gamma$, together with Lemma \ref{l:MSOtoOmega}, there exists a map $$F: Q_2 \rightarrow MSO(n-1)\wedge M(\Z_k,1)$$ such that $F^*(u) = x_{|Q_2}$, where $u \in H^n(MSO(n-1)\wedge M(\Z_k,1))$ is the mod $k$ Thom class; $F$ is smooth and by a two-stage transversality procedure as in Theorem \ref{t:criterionmodk} we obtain 
\[
    F^{-1}(BSO(n-1) \times \{\hat{e_2}\})  \cup F^{-1}(BSO(n-1) \times \{\hat{e}_1\})=\Gamma,
\]
which is the smooth part of $S$. Denote $$Q_5:=\mathcal{M} \setminus \mathcal{K}^{m-5}.$$
Given the $n$-dimensional cohomology class $x \in H^{n}(\mathcal{M})$ which is the Poincaré dual of $\tau$, we consider its restriction to $Q_5$, that is $x_{|Q_5} \in H^n(Q_5)$; note that $x_{|Q_5}$ can be represented by a continuous map $$g:Q_5 \rightarrow K(\mathbb{Z}_k,n)$$ (in a suitable homotopy class of continuous maps) pulling-back the fundamental class of $K(\mathbb{Z}_k,n)$ to itself, \textit{i.e.} $g^*(\iota)=x_{|Q_5}.$ Since $Q_5$ has the homotopy type of an $n+4$ dimensional complex, we can apply the $4$-equivalence of Proposition \ref{p:4-equivalence} to obtain a lift from $K(\mathbb{Z}_k,n)$ to $\Omega^{\infty}\Sigma^{n}(MSO\wedge M\Z_k)$ pulling back the mod $k$ Thom class to $x_{|Q_5}.$


In addition, since $n \geq\min\{m+1,5\}$, the $2n-1$-equivalence between $MSO(n-1) \wedge M(\Z_k,1)$ and $\Omega^{\infty}\Sigma^{n}(MSO\wedge M\Z_k)$ from Lemma \ref{l:MSOtoOmega} allows us to lift the map again and obtain a map

$$f: Q_5 \rightarrow MSO(n-1) \wedge M(\Z_k,1)$$

pulling back the mod $k$ Thom class to $x_{|Q_5}.$ By Theorem \ref{t:criterionmodk}, we can perturb $f$ to a new smooth map so that by two-stage transversality we obtain a properly embedded $\Z_k$-smooth $m$-dimensional submanifold; denote it as $\Sigma.$ Moreover, $\Sigma$ represents the relative Poincaré dual of $x_{|Q_5}$, which equals $i_*(\tau) \in H^{BM}_m(Q_5)$ where $i_*:H_m(\mathcal{M}) \rightarrow H^{BM}_m(Q_5).$ 

Moreover, note that we have immediately obtained a mod $k$ integral current $N:=\llbracket \Sigma \rrbracket_{\modk}$ such that $N\res \mathcal{M}\setminus \mathcal{K}^{m-5}$ is induced by some embedded $\Z_k$-smooth submanifold $\Sigma$ (with multiplicity one). Moreover, $\spt(\partial N) \subset \mathcal{K}^{m-5}$ and in particular by Federer's support theorem, $\partial N=0$, namely $N$ is a mod $k$ integral cycle.

Now consider the two maps $F: Q_2 \rightarrow MSO(n-1) \wedge M(\Z_k,1)$ and $f: Q_5 \rightarrow MSO(n-1) \wedge M(\Z_k,1)$ that by two-step transversality give us $\Gamma$ and $\Sigma \cap Q_2$ respectively. If we consider the restriction of $f$ to $Q_2 \subset Q_5$, we obtain a new map $f|_{Q_2}: Q_2 \rightarrow MSO(n-1) \wedge M(\Z_k,1)$ that pulls-back the universal Thom class to $x_{|Q_2}$. 

We observe that $Q_2$ has the homotopy type of an $n+1$-complex, so that by Proposition \ref{p:4-equivalence} we can conclude that $F$ and $f_{|Q_2}$ are homotopic: the homotopy can be taken smooth by \cite[Lemme IV.5]{Thom54}. In particular, we define the smooth homotopy $H:[0,1] \times Q_2$ such that $H(0,x)= f_{|Q_2}(x)$ and $H(1,x)= F(x)$. 

In a small smooth neighborhood $\mathcal{C}$ of $\mathcal{K}^{m-2}$ in $Q_2$, which we identify with $\partial Q_2\times (0,1]$, we may then glue the maps $f$ and $F$ together. Using the notation $x=(y,s)\in \partial Q_2\times (0,1] = \mathcal{C}$ and after defining a smooth function $\varphi:[0,1] \to [0,1]$ which is identically equal to $0$ in a neighborhood of $0$ and identically equal to $1$ in a neighborhood of $1$, we set  
\begin{equation}\label{e:homotopy}
\widehat{f}(x):=\begin{cases}
       F(x)  & \quad if  \, \, x\in Q_2\setminus \mathcal{C},\\    
       H\left(x,\varphi (s)\right)  & \quad if\,\,  x\in \mathcal{C},   \\
        f(x)  & \quad if\,\, x\in Q_5\setminus Q_2  
\end{cases}\end{equation}

Note that we can make $\widehat{f}$ transverse locally while maintaining the property that it coincides with $F(x)$ on $\mathcal{M}\setminus U(\mathcal{K}^{m-2})$, where $U(\mathcal{K}^{m-2})$ is a small smooth neighborhood of $\mathcal{K}^{m-2}$ and without changing the homotopy class.

Consider now the submanifold $\Sigma'$ of $\mathcal{M}\setminus \mathcal{K}^{m-5}$ which is obtained by the two-step transversality procedure of $\widehat{f}$ on $Q_5$. This is a $\Z_k$-smooth submanifold coinciding with $\Gamma$ on $\mathcal{M}\setminus U(\mathcal{K}^{m-2})$.

Moreover, $R= \llbracket \Sigma'\rrbracket_{\modk}$ is a mod $k$ integral current with finite mass and such that $\spt (\partial R)\subset \mathcal{K}^{m-5}$; in particular it is a cycle by Federer's support theorem. Observe also that $R-S$ is supported, by construction, in $U (\mathcal{K}^{m-2})$, which is homotopy equivalent to $\mathcal{K}^{m-2}$, and thus has trivial $m$-homology. In particular $R-S$ is a mod $k$ boundary, namely $R$ and $S$ belong to the same mod $k$ homology class. 

We now apply Proposition \ref{p:squash} to $S$ and $R$, noticing that the $\varepsilon_d$ in Proposition \ref{p:squash} is a parameter to be chosen in terms of the $\varepsilon$ of the statement of Theorem \ref{t:1}, and the $\eta_d$ in Proposition \ref{p:squash} is $\delta'_c$ here. This gives us a parameter $\delta_d$, which depends on $\varepsilon_d$ and $\delta'_c$. In turn we impose that $\tilde{\delta} \leq \delta_d$ so that we can apply Proposition \ref{p:squash}. Since $\varepsilon_d$ will be specified only in terms of $\mathbb{M}(T)$ and of $\varepsilon$ in the statement of Theorem \ref{t:1}, while $\delta'_c$ depends on $\varepsilon_c$, which will also be specified only in terms of $\mathbb{M}(T)$ and $\varepsilon$ in the statement of Theorem \ref{t:1}, the parameter $\tilde{\delta}$ can be taken smaller than $\delta_d$. We can then find a mod $k$ current $R':=\Phi_\sharp R$ for a smooth diffeomorphism $\Phi$ isotopic to the identity such that  
\begin{align}
\mathbb{M} (R') \leq (1+\varepsilon_d)\, \mathbb M (S) \leq (1+\varepsilon_d) (\mathbb{M} (T) + 3\varepsilon_c)\,\nonumber .
\end{align}
We therefore conclude that $R'$ is homologous to $R$, hence to $S$, and therefore to $T$. Moreover, if we choose
\begin{align}
&\varepsilon_d \,(3+\mathbb{M} (T)) < \frac{\varepsilon}{2} \qquad \mbox{and} \qquad
\,3\varepsilon_c < \frac{\varepsilon}{2}\,\nonumber ,
\end{align}
then $\mathbb{M} (R') \leq \mathbb{M} (T) + \varepsilon$. This establishes claim (1) of the theorem. Moreover,
\begin{align}
\mathbb{F} (T-R') &\leq 3\varepsilon_c + \mathbb{F} (S-R') 
\leq 3\varepsilon_c + C (\varepsilon_d
\,\mathbb{M} (S) + 2\|S\| (B_{\delta'_c}(\mathcal{K}^{m-2})))^{\frac{m+1}{m}}\nonumber\\
&\leq 3\varepsilon_c + C (\varepsilon_d (\mathbb{M} (T) +\varepsilon_c) + 6\varepsilon_c)^{\frac{m+1}{m}}\,\nonumber.
\end{align}
Therefore one may clearly make a suitable choice of $\varepsilon_d$ and $\varepsilon_c$ depending only on $\mathbb{M} (T)$ and $\varepsilon$ in order to show that $\mathbb{F} (T-R') \leq \varepsilon$, thus establishing claim (2). 

The proof of claim $(3)$ of Theorem \ref{t:1} follows the same lines; by assumption we know that $\tau$ is represented by an embedded $\Z_k$-smooth submanifold $\Sigma$ and hence, by Theorem \ref{t:criterionmodk} (together with Lemma \ref{l:MSOtoOmega}) there exists a map $g:\mathcal{M} \rightarrow MSO(n-1) \wedge M(\Z_k,1)$ pulling back the mod $k$ Thom class $u \in H^n(MSO(n-1) \wedge M(\Z_k,1))$ to the Poincaré dual of $\tau$. Substituting in the previous steps the map $f$ with this new map $g$, defined over the whole ambient space $\mathcal{M}$, and defining a similar homotopy to the one in \eqref{e:homotopy}, the result follows by applying Proposition \ref{p:squash} to $S$ and $\llbracket \Sigma \rrbracket_{\modk}$, where $S$ is the mod $k$ integral cycle denoted by $P'$ in Proposition \ref{p:poly_approx_prescribedsing}.

\end{proof}

We now deduce Theorem \ref{t:Mooreapproximation} as a corollary of Proposition \ref{p:poly_approx_prescribedsing} and the homotopy properties of the Moore space $M(\Z_k,n)$ for $k$ odd and in high codimensions $n\geq 3.$

\begin{proof}[Proof of Theorem \ref{t:Mooreapproximation}]
By Proposition \ref{p:poly_approx_prescribedsing}, we obtain in particular a map $$\tilde{f}: \mathcal{M}\setminus \mathcal{K}^{m-2} \rightarrow S^{n-1} \wedge M(\Z_k,1) \simeq M(\Z_k,n) $$ pulling back the fundamental class to the restriction of the Poincaré dual of $\tau$ to $\mathcal{M}\setminus \mathcal{K}^{m-2}$. On the other hand, by assumption we also have a map $$g:\mathcal{M} \rightarrow M(\Z_k,n)$$ pulling back the fundamental class on the Poincaré dual of $\tau.$ 

For $k$ odd we have that $\pi_{n}(M(\Z_k,n))\simeq \Z_k$ and $\pi_{n+1}(M(\Z_k,n))\simeq 0,$ \emph{cf.} \cite[Section 2]{Moorehomotopygroups}\footnote{The result is derived from computing the stable homotopy groups first. Then, since $M(\Z_k,n)$ is $n-1$-connected, by Freudenthal suspension theorem we obtain that the suspension is an isomorphism in degrees $n+1\leq 2n-2$, that is $n \geq 3$, see \cite[Corollary 9.4]{Baues}.}. Hence, the map $$h: M(\Z_k,n) \rightarrow K(\Z_k,n) $$representing the fundamental class is at least an $n+1$-equivalence when $k$ is odd and $n\geq 3.$ This allows $g_{|\mathcal{M}\setminus \mathcal{K}^{m-2}}$ to be homotopic to $\tilde{f}$ on $\mathcal{M}\setminus \mathcal{K}^{m-2}$, and hence conclude following the same lines of the proof of Theorem \ref{t:1}.    
\end{proof}

\section{Sharpness of the codimension 5 singular set}\label{s:sharpness}

In this section, we will demonstrate that the codimension 5 construction in Theorem \ref{t:1} is optimal (in the sense of the dimension of the singular set) in stable regimes.

We start by recalling the example of a mod $k$ homology class of dimension 6 which is not realizable by means of a submanifold. We will base this section on Kreck's theory of stratifolds and Sullivan's theory of geometric resolution of singularities. We recall that in \cite{SullivanLiverpool}, Sullivan described a procedure to resolve singularities of geometric cycles. It consists of a blow-up process to resolve singularities by inductively attaching generalized handles. He showed that for a geometric cycle $T$ with singularity locus $S_T$ of dimension $s$, the obstruction for resolving the singularity is a cycle in $H_s(S_T , \Omega_r)$, with values in a suitable cobordism group of dimension $r=\operatorname{dim} T -s -1$.

Instead, the theory of stratifolds provides a bordism theory $S H_*$ with a natural isomorphism to integral homology. In particular, a stratifold $S$ of dimension $m$ is a topological space with a sheaf of functions which divides the topological space into strata $S_0, \cdots, S_m$, where each stratum $S_i$ is an $i$-dimensional manifold. The singular part of the stratifold $S$ is given by the union $\sing(S):=\bigcup_{i=0}^{m-1} S_i$. They are useful since they allow for geometric descriptions of cycles. Recently, the theory of stratifolds has been generalized by \cite{3tizi} to $\Z_k$-coefficients, providing in the same article a geometric description of the singularities of mod $k$ homology classes that cannot be represented as images of $\Z_k$-manifolds. We briefly recall their results below, referring to \cite{3tizi} for an introduction to stratifolds and their $\Z_k$-analogs.

In particular, in \cite[Section 8]{3tizi} it is shown that the nonzero generator of $B\Z_3$ (where we recall that $B\Z_3$ is the classifying space of $\Z_3$, which can be realized by the infinite lens space $L^{\infty}_3$) $$\bar{z}_6\in H_6(B\Z_3, \Z_3)\simeq \mathbb Z_3$$ cannot be represented by a map from a $\Z_3$-manifold. In particular, this class can be represented by a $6$-dimensional mod 3 integral cycle $T$ with a $1$-dimensional stratum of non-removable singularities (which is a circle $S^1$) whose neighborhood is homeomorphic to a product $$\mathcal{S}_T \times C(\mathbb{C}\mathbb{P}^2),$$ where $C(\mathbb{C}\mathbb{P}^2)$ denotes the cone over $\mathbb{C}\mathbb{P}^2$. As in the oriented setting, this innate nature of these singularities turns out to be intrinsically linked to the well-known fact that $\mathbb{C}\mathbb{P}^2$ does not bound any compact oriented smooth $5$-dimensional manifold, as observed in \cite{SullivanLiverpool}.

In summary, we adapt this information and construct the following example, yielding the desired sharpness for the dimension of the topological singular set in Theorem \ref{t:1}.

\begin{example}\label{e:modknon}
Let $L:=L^d_3$ denote the $d$-dimensional lens space $S^d/\Z_3$. Assume $d$ is an odd integer and $d\geq 13$ (so that we are in the stable regime). Note that the inclusion into the infinite lens space gives a map $i: L \rightarrow B\mathbb{Z}_3$ inducing the isomorphism $i_*: H_6(L,\Z_3) \rightarrow H_6(B\mathbb{Z}_3, \mathbb Z_3)$. Then the generator $z= (i_*)^{-1}(\bar{z})$ is not representable by a $\Z_3$-manifold.

\end{example}

\begin{thm}\label{t:Thom_innatelysingular}
    Let $z \in H_6(L,\mathbb{Z}_3)$ be the mod 3 homology class of Example \ref{e:modknon} and fix a smooth triangulation $\mathcal{K}$ of $L$. Then it is impossible to find a representative $\Sigma$ for $z$ which is a $\Z_3$-smooth submanifold in the complement of the $0$-dimensional skeleton $\mathcal{K}^0$ of $\mathcal{K}$.
\end{thm}

\begin{proof}
  Denote by $T$ the 6-dimensional mod $k$ integral cycle representing $z$, and consider its singularity locus $\mathcal{S}_T$. Towards a proof by contradiction, assume that the cycle is a substratified set of a Whitney stratification of $L$ which only intersects the zero-skeleton $\mathcal{K}^0$ of a triangulation compatible with the stratification. For each 0-dimensional simplex in the cycle we consider its link, which is a 5-dimensional $\Z_3$-manifold. Since the obstruction to the resolution of singularities is an element of $$\omega_{T} \in H_0(\mathcal{S}_T, \Omega_5)$$ and the 5-dimensional $\Z_3$-cobordism group vanishes, by the $\Z_3$-manifold version of \cite[Theorem D]{SullivanLiverpool} the cycle can be resolved by blow-up to the lower-dimensional stratum, obtaining full resolution of singularities. This is clearly a contradiction.
\end{proof}

\section{On the setwise approximation of cycles}\label{s:setwise}

The aim of this section is to investigate the possibility of locally approximating the support of any integral or mod $k$ integral cycle by smooth or $\mathbb{Z}_k$-smooth submanifolds, respectively. We first construct an example of an area-minimizing cycle $T$ in a nontrivial integral homology class (admitting a smooth embedded representative and hence, by \cite[Theorem 1.1 (3)]{ABCD}, admitting a \emph{global} smooth approximation) of an oriented Riemannian manifold $\mathcal{M}$ whose support cannot be \emph{locally} approximated by a smooth submanifold, \emph{cf.} \cite[page 21]{AlmgrenBrowder88-91} and \cite[Chapter 4]{tesi} for preliminary discussions on this. We then adapt the strategy of this example to show that the same phenomenon can be replicated for mod $k$ minimizers of nontrivial $\Z_k$-homology classes with $\Z_k$-smooth manifold representatives. In particular, Example \ref{e:support1} corrects and formalizes previous attempts proposed by Almgren and Browder in \cite[page 21]{AlmgrenBrowder88-91} and \cite{Almgren90-93}, though using the same idea based on the so-called \emph{thickening} strategy developed by Thom in his resolution of Steenrod's problem in \cite[Chapitre III]{Thom54}.

We start by recalling Thom's example of an integral homology class of dimension 7 in a real analytic closed oriented manifold of dimension 14 which is not realizable as the continuous image of a submanifold, \textit{cfr.} \cite[Théorème III.9]{Thom54}.

\begin{example}[Thom]\label{ex:ThomLensSpaces}
For $i=1,2$ consider the lens space $L_i:=S^7/\mathbb{Z}_3$, which is the orbit space of the $7$-sphere with the free action of $\mathbb{Z}_3$ generated by the rotation. Let $v_i$ be the generator of $H^1(L_i,\mathbb{Z}_3)\simeq \mathbb{Z}_3$ and call $u_i=\beta_3(v_i) \in H^2(L_i,\mathbb{Z}_3)\simeq \mathbb{Z}_3$, where $\beta_3$ is the mod 3 Bockstein operation. Consider the real analytic, closed and oriented $14$-dimensional manifold $L:=L_1\times L_2$ and the following cohomology class, where the powers and $\cdot$ denote the cup product (viewing $H^*(L_i)$ as being embedded in $H^*(L)$): $$y=u_1 \cdot v_2 \cdot(u_2)^2 - v_1 \cdot (u_2)^3 \in H^7(L,\mathbb{Z}_3).$$ Note that $y$ is actually the reduction mod $3$ (denoted $\theta_3$) of an integral cohomology class, since $y=\beta_3(v_1 \cdot v_2 \cdot (u_2)^2)$ and hence $y=\theta_3(x)$ with $x \in H^7(L,\mathbb{Z})$ given by $x =\beta^*(v_1 \cdot v_2 \cdot (u_2)^2)$, where $\beta^*$ is the Bockstein operation, \textit{cf.} \cite[Remark 1.6]{ABCD}. Denoting by $z \in H_7(L,\mathbb{Z})$ its Poincaré dual homology class, we see that $z$ cannot be realized in $L$ by an embedded smooth submanifold since$$St_3^5(x)=\beta^*\circ \mathcal{P}_3^1(y)=\beta^*((u_1)^3 \cdot v_2 \cdot (u_2)^2) = (u_1 \cdot u_2)^3 \neq 0,$$
where $\mathcal{P}_3^1$ is the first Steenrod power for the prime 3.
\end{example}

\begin{example}\label{e:support1}
Let $P$ be a representative in Thom's 7-dimensional integral homology class $z \in H_7(L,\Z)$ of Example \ref{ex:ThomLensSpaces} that is assumed, without loss of generality, to have a fundamental class and multiplicity one on its top dimensional simplices\footnote{For instance by applying a similar construction of the codimension 2 smooth approximation theorem in \cite[Proposition 4.1]{ABCD}.}, so that we have a map $h: P \rightarrow L$ with $h_*[P]=z$. Let $$f: P \rightarrow S^7$$ be a map of degree 1, which can be constructed from $P$ simply by collapsing the complement of the interior of an $m$-dimensional face to the point at infinity of $S^7$. Let $\Sigma_f$ denote the mapping cylinder of $f$, with $\Sigma_0=P$ and $\Sigma_1=S^7$. We embed $\Sigma_f$ in a high-dimensional Euclidean space $\R^N$ in such a way $\Sigma_1$ is a smoothly embedded submanifold of $\R^N$. Take a compact smooth regular neighborhood $U$ of $\Sigma_f$ in $\R^N$ and let the ambient manifold $\mathcal{M}:= DU$ be defined as the double of $U$ along its smooth boundary $\partial U$. As the 7-dimensional integral homology class $\tau \in H_7(\mathcal{M},\Z)$ we take the image of that of $P$ (or, equivalently, of $\Sigma_1$) included in $\mathcal{M}$. By construction, $\tau$ is nontrivial in $\mathcal{M}$ since $f$ has degree one and $\mathcal{M}$ retracts onto $\Sigma_f$, and hence onto $\Sigma_1$; in addition, it admits a smooth embedded representative, given by the submanifold $\Sigma_1=S^7$. 

Fix a metric $(\mathcal{M},g)$ and let $\delta >0$ be a radius such that $B_{5\delta}(P) \subset \text{int } U$ and $B_{5\delta}(P)$ retracts onto $P$; denote by $r : B_{5\delta}(P) \rightarrow P$ the retraction map. We now change the metric in $\mathcal{M}$ to be very small inside the $\delta$-neighborhood of $P$, while leaving it unchanged outside the $3\delta$-neighborhood of $P$ in $\mathcal{M}$. More precisely, for $\varepsilon \in (0,1)$ we choose a smooth metric $g_{\varepsilon}$ such that $g_{\varepsilon}= \varepsilon g$ on $B_{\delta}(P)$ and $g_{\varepsilon}=g$ on $\mathcal{M}\setminus B_{3\delta}(P)$. In particular, we choose $\varepsilon$ sufficiently small so that $$\mathbb{M}_{g_{\varepsilon}}(P)\leq \frac{1}{4} \omega_7 \delta^7,$$ where $\mathbb{M}_{g_{\varepsilon}}(P)$ denotes the mass of $P$ with respect to $g_{\varepsilon}$ and $\omega_7$ the volume of the 7-dimensional unit ball. 

Let $Q$ be a minimizer in $\tau$ for $\Mbb_{g_{\varepsilon}}$. Suppose now that we can find $x \in \text{spt }Q \setminus B_{4\delta}(P).$ Since $B_{\delta}(x) \cap B_{3\delta}(P) = \emptyset$, then $Q$ is a local minimizer for $h$ in $B_{\delta}(x) $. Hence, by the monotonicity formula for mass ratios, we have that $$\mathbb{M}_{g_{\varepsilon}}(Q) \geq \mathbb{M}\left(Q \res B_{\delta}(x)\right) \geq \frac{1}{2} \omega_7 \delta^7 > \mathbb{M}_{g_{\varepsilon}}(P),$$ contradicting the minimality of $Q$. Therefore we conclude that $\text{spt }Q \subset B_{4\delta}(P),$ that is we are forcing any minimizer $Q$ to have support in a small neighborhood of $P$, and in particular inside $B_{5\delta}(P).$

By \cite[Theorem 1.1 (3)]{ABCD} we obtain a smooth embedded submanifold $\Sigma$ in $\tau$ close in the flat topology to $Q$, and such that $\mathbb{M}_{g_{\varepsilon}}(\Sigma) \leq \mathbb{M}_{g_{\varepsilon}}(Q) + \tilde{\varepsilon}$. If we could also conclude vicinity of the supports, that is $\Sigma \subset B^{g_{\varepsilon}}_{\delta}(\spt Q) \subset B_{5\delta}(P)$, then we would reach a contradiction. In fact, since $\llbracket \Sigma \rrbracket \in \tau$ and since the homomorphism $H_{7}(B_{5\delta}(P), \Z) \rightarrow H_{7}(\mathcal{M},\Z)$ is injective, then $\llbracket \Sigma \rrbracket$ is homologous to $P$ in $B_{5\delta}(P)$, and hence the retraction map would imply $r_*[\Sigma]=[P]$, deriving the contradiction that $z=h_*[P]=(h \circ r)_*[\Sigma]$ is represented as continuous image of the smooth manifold $\Sigma$.
\end{example}

The same thickening strategy can be adapted to the case of mod $k$ minimizing cycles, providing an example of an area-minimizing mod $k$ integral cycle in a nontrivial homology class (admitting a $\Z_k$-manifold representative) that cannot be approximated by a $\mathbb{Z}_k$-smooth manifold.

\begin{example}\label{e:support2}
Consider the $9$-skeleton $K:=K(\mathbb{Z},3)^{(9)}$ of the Eilenberg-MacLane space $K(\mathbb{Z},3)$, whose homology is isomorphic to that of $K(\mathbb{Z},3)$ up to degree $8$. The nontrivial class $z \in H_8(K,\Z_3)$ is not representable as image of a $\Z_3$-manifold; let $P$ be a mod 3 cycle representing $z$. Up to choosing a suitable representative $P$, let $$f: P \rightarrow S_3^8$$ be a map of degree 1 into the 8-dimensional $\mathbb{Z}_3$-sphere, and let $\Sigma_f$ denote the mapping cylinder of $f$, with $\Sigma_0=P$ and $\Sigma_1=S_3^8$. We now proceed analogously as in Example \ref{e:support1}: we now embed $\Sigma_f$ into some high-dimensional Euclidean space $\R^N$ so that $\Sigma_1$ is an embedded $\Z_3$-smooth submanifold and take a smooth regular neighborhood $U$. Take $\mathcal{M}$ the double manifold of $U$ along $\partial U$. Consider now $\tau \in H_8(\mathcal{M},\Z_3)$ the 8-dimensional homology class corresponding to $\Sigma_0$ (or, equivalently, to $\Sigma_1$) and note that $\tau$ is nontrivial in $\mathcal{M}$ (since $\mathcal{M}$ retract onto $U$, and hence onto $\Sigma_1$). Moreover, it admits a smooth embedded $\mathbb{Z}_3$-manifold representative, given by $\Sigma_1$. We may change the metric in $\mathcal{M}$ if necessary to be very small near $\Sigma_0$, and use the monotonicity formula for mass ratios to ensure that the support of any area-minimizing mod 3 integral cycle $T$ in $\tau$ must lie very close to $\Sigma_0$. Again, it is not possible for a $\Z_3$-submanifold representing $\tau$ to be contained in a small neighborhood of $\Sigma_0$, since this would contradict the fact that the homology class of $P$ is not representable as image of a $\Z_3$-manifold.
\end{example}

It is interesting to note -- and, ultimately, this observation goes back to Thom \cite[Chapitre III]{Thom54} -- that the thickening strategy does not give a contradiction for mod 2 cycles, since every mod 2 homology class of a CW-complex can be represented as the continuous image of a smooth manifold, see \cite[Théorème III.2]{Thom54}. Since there exist mod 2 homology classes of smooth manifolds with no embedded (or even immersed) reprentative, it is nevertheless expected that a similar argument could be carried out by exploiting obstructions to embeddability (or immersability) instead.

\end{document}